\documentclass[11pt,english]{article}

\usepackage[margin = 2.5cm]{geometry}

\usepackage{amsthm}
\usepackage{amsmath}
\usepackage{amssymb}
\usepackage{mathtools}
\usepackage{graphicx,enumerate}
\usepackage[hidelinks, bookmarks = false]{hyperref}
\usepackage{enumitem}
\usepackage{subcaption}

\usepackage{caption}
\newlist{steps}{enumerate}{1}
\setlist[steps, 1]{label = Step \arabic*:}

\usepackage[square,sort,comma,numbers]{natbib} 
\usepackage{setspace}
\usepackage{cleveref}
\theoremstyle{plain}

\newtheorem*{thm*}{Theorem}
\newtheorem{thm}{Theorem}[section]
\crefname{thm}{Theorem}{Theorems}
\Crefname{thm}{Theorem}{Theorems}

\newtheorem*{lem*}{Lemma}
\newtheorem{lem}[thm]{Lemma}
\crefname{lem}{Lemma}{Lemmas}
\Crefname{lem}{Lemma}{Lemmas}

\newtheorem*{claim*}{Claim}
\newtheorem{claim}[thm]{Claim}
\crefname{claim}{Claim}{Claims}
\Crefname{claim}{Claim}{Claims}

\crefname{prop}{Proposition}{Propositions}
\Crefname{prop}{Proposition}{Propositions}

\newtheorem{cor}[thm]{Corollary}
\crefname{cor}{Corollary}{Corollaries}
\Crefname{cor}{Corollary}{Corollaries}

\newtheorem{conj}[thm]{Conjecture}
\crefname{conj}{Conjecture}{Conjectures}
\Crefname{conj}{Conjecture}{Conjectures}

\crefname{qn}{Question}{Questions}
\Crefname{qn}{Question}{Questions}

\newtheorem{obs}[thm]{Observation}
\crefname{obs}{Observation}{Observations}
\Crefname{obs}{Observation}{Observations}

\crefname{ex}{Example}{Examples}
\Crefname{ex}{Example}{Examples}

\theoremstyle{definition}

\crefname{prob}{Problem}{Problems}
\Crefname{prob}{Problem}{Problems}

\crefname{defn}{Definition}{Definitions}
\Crefname{defn}{Definition}{Definitions}

\theoremstyle{remark}

\crefname{rem}{Remark}{Remarks}
\Crefname{rem}{Remark}{Remarks}

\usepackage{xpatch}
\xpatchcmd{\proof}{\itshape}{\normalfont\proofnamefont}{}{}
\newcommand{\proofnamefont}{}
\renewcommand{\proofnamefont}{\bfseries}

\usepackage{framed}

\newcommand{\remove}[1]{}

\newcommand{\ceil}[1]{
    \lceil #1 \rceil
}
\newcommand{\floor}[1]{
    \lfloor #1 \rfloor
}

\newcommand{\Q}{\mathcal{Q}}
\newcommand{\G}{\mathcal{G}}
\newcommand{\C}{\mathcal{C}}
\newcommand{\FF}{\mathcal{F}}
\newcommand{\eps}{\varepsilon}
\newcommand{\sub}{\mathrm{sub}}
\newcommand{\sbb}{\mathrm{sb}}
\newcommand{\ext}{\mathrm{ext}}
\newcommand{\final}{\mathrm{final}}

\newcommand{\triangulated}{triangulated}

\DeclareMathOperator{\Forb}{Forb}
\newcommand{\Forbstar}{\Forb^*}

\title{Chi-boundedness of graphs containing no cycles with $k$ chords}

\author{
	Joonkyung Lee\thanks{Department of Mathematics, Yonsei University, Seoul, South Korea.
E-mail: \texttt{joonkyunglee@yonsei.ac.kr}.
	Research supported by the National Research Foundation of Korea (NRF) grant funded by the Korea government MSIT NRF-2022R1C1C1010300, Samsung STF Grant SSTF-BA2201-02, and IBS-R029-C4.} 
	\and
	Shoham Letzter\thanks{
		Department of Mathematics, 
		University College London, 
		London WC1E~6BT, UK. 
		Email: \texttt{s.letzter}@\texttt{ucl.ac.uk}. 
		Research supported by the Royal Society.
    }
    \and
	Alexey Pokrovskiy\thanks{
	    Department of Mathematics, 
	    University College London,
	    London WC1E~6BT, UK.
	    Email: \texttt{a.pokrovskiy}@\texttt{ucl.ac.uk}.
	}
}

\begin{document}

\date{}
\maketitle

\begin{abstract}
    \noindent
    We prove that the family of graphs containing no cycle with exactly $k$-chords is $\chi$-bounded, for $k$ large enough or of form $\ell(\ell-2)$ with $\ell \ge 3$ an integer. This verifies (up to a finite number of values $k$) a conjecture of Aboulker and Bousquet (2015).

\end{abstract}

\section{Introduction} \label{sec:intro}
    The \emph{clique number} of a graph $G$, denoted $\omega(G)$, is the size of its largest clique in $G$. The \emph{chromatic number} of $G$, denoted $\chi(G)$, is the minimum number of colours in a \emph{proper vertex-colouring} of $G$, which is a colouring of the vertices where adjacent vertices have distinct colours. 
    It is easy to see that $\chi(G) \ge \omega(G)$ for every graph $G$, but the converse is far from the truth. Indeed, the chromatic number cannot be upper-bounded by a function of the clique number. This can be seen, for example, through a construction due to Mycielski \cite{mycielski1955coloriage} that provides a family of triangle-free graphs whose chromatic number is unbounded. 
    
    In 1987, Gy\'arf\'as \cite{gyarfas1987problems} proposed to study families of graphs for which the chromatic number can be upper-bounded in terms of the clique number. More precisely, Gy\'arf\'as called a family of graphs $\G$ \emph{$\chi$-bounded} if there is a function $f$ such that $\chi(G) \le f(\omega(G))$ for every $G \in \G$. 
    
    For a graph $F$, denote by $\Forb(F)$ the family of graphs that contain no induced copy of $F$. 
    This is a particularly interesting class of graphs for studying $\chi$-boundedness, as it leads us to various examples and conjectures. Namely, one may ask:
    for which graphs $F$ is $\Forb(F)$ $\chi$-bounded?
    If $F$ contains a cycle, $\Forb(F)$ is not $\chi$-bounded; indeed, this follows from the existence of graphs with arbitrarily large chromatic number and girth (where the \emph{girth} of a graph is the length of its shortest cycle), due to Erd\H{o}s \cite{erdos1959graph}.
    Gy\'arf\'as proved \cite{gyarfas1987problems} that $\Forb(F)$ is $\chi$-bounded when $F$ is a path or a star. 
    An intriguing conjecture regarding $\chi$-boundedness, due to Gy\'arf\'as \cite{gyarfas1975ramsey} and Sumner \cite{sumner1981subtrees}, asserts that $\Forb(F)$ is $\chi$-bounded for every forest $F$.
    If true, this would solve the above question regarding graphs $F$ for which $\Forb(F)$ is $\chi$-bounded. 
    The conjecture is known for some special cases, including all trees of radius $2$ \cite{kierstead1994radius} and some trees of radius $3$ \cite{kierstead2004radius}, but is widely open in general.
    
    For a graph $F$, denote by $\Forbstar(F)$ the family of graphs that do not contain an induced copy of a \emph{subdivision} of $F$, where a subdivision of $F$ is a graph obtained by replacing edges of $F$ by internally disjoint paths. 
    Scott \cite{scott1997induced} proved the following weakening of the G\'yarf\'as--Sumner conjecture: $\Forbstar(F)$ is $\chi$-bounded for every forest $F$. He also conjectured that $\Forbstar(F)$ is $\chi$-bounded for every graph $F$, but this turned out to be false \cite{pawlik2014triangle,chalpion2016restricted}. At the best of our knowledge, there seems to be no conjectured answer to the question that asks for which graphs $F$ is $\Forbstar(F)$ $\chi$-bounded.
    
    The fact that $\Forbstar(F) \subseteq \Forb(F)$ for every graph $F$ makes it natural to consider an `interpolation' between the two classes to ask an analogous question.
    More precisely, 
    fix an edge subset $E$ of $F$ and consider the class $\FF$ of graphs that contain no induced copy of a graph obtained from $F$ by only subdividing the edges in $E$, while leaving the other edges unchanged.
    For example, taking $F$ to be the graph obtained by adding a diagonal to a $4$-cycle and letting $E$ be set of edges in the cycle, the family $\FF$ obtained as above is the family of graphs that do not have a cycle with exactly one \emph{chord}; given a graph $G$, a \emph{chord} in a cycle $C$ in $G$ is an edge that joins two vertices in $C$ which are not adjacent in $C$.
    Trotignon and Vu\v{s}kovi\'c \cite{trotignon2010structure} showed that this family of graphs with no cycle with exactly one chord is $\chi$-bounded (in fact, they give a detailed description of the structure of such graphs). Inspired by this, Aboulker and Bousquet \cite{aboulker2015excluding} suggested to study the family $\C_k$ of graphs that do not contain a cycle with exactly $k$ chords. They conjecture that $\C_k$ is $\chi$-bounded for every $k \ge 1$, and prove it for $k \in \{2, 3\}$. Note that \cite{trotignon2010structure} establishes the conjecture for $k = 1$. Our main result is to prove the Aboulker--Bousquet conjecture for sufficiently large $k$.
    
	\begin{thm} \label{thm:main}
	    For every large enough $k$,\footnote{Concretely, it suffices to take $k \ge 10^{14}$.} there is a function $f_k$ such that, if $G$ is a graph with no cycle with exactly $k$ chords, then $\chi(G) \le f_k(\omega(G))$.
	\end{thm}
	
    A \emph{$k$-wheel} is a cycle with an additional vertex that is adjacent to exactly $k$ vertices in the cycle, and a \emph{wheel} is a $3$-wheel. It is easy to see that a $k$-wheel contains a cycle with exactly $\ell$ chords, for any $\ell \le k-2$ (see \Cref{fig:wheel}).
    
    \begin{figure}[h]
        \centering 
        \begin{subfigure}[b]{.4\textwidth}
            \centering
            \includegraphics[scale = 1.1]{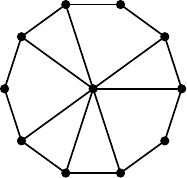}
            \vspace{.5cm}
            \caption*{A $7$-wheel}
        \end{subfigure}
        \hspace{1cm}
        \begin{subfigure}[b]{.4\textwidth}
            \centering
            \includegraphics[scale = 1.1]{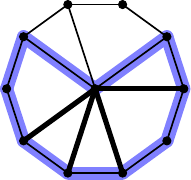}
            \vspace{.5cm}
            \caption*{A cycle with four chords in a $7$-wheel}
        \end{subfigure}
        \caption{Wheels}
        \label{fig:wheel}
    \end{figure}
    
    It is thus quite natural, given our discussion of graphs with cycles with a given number of chords, to also consider graphs avoiding wheels. Trotignon \cite{trotignon2013perfect} conjectured that the class of wheel-free graphs is $\chi$-bounded. This conjecture is now known to be false (see Davies \cite{davies2023triangle}), but a weaker version of it was proved by Bousquet and Thomass\'e \cite{bousquet-thomasse}; they showed that, for every integer $\ell \ge 1$, the class of graphs with no induced wheel or an induced $K_{\ell, \ell}$ is $\chi$-bounded. They also proved a similar result for $k$-wheels, with an additional triangle-free requirement; namely, they showed that, for every $\ell$, the class of graphs with no induced $k$-wheel, $K_{\ell,\ell}$ or triangle is $\chi$-bounded. Since a $2\ell$-cycle in $K_{\ell, \ell}$ has $\ell(\ell - 2)$ chords, taking $k = \ell(\ell - 2) + 2$ shows that the family of triangle-free graphs with no cycle with exactly $\ell(\ell-2)$ chords is $\chi$-bounded. We strengthen the latter statement by removing the triangle-free requirement.
	
	\begin{thm} \label{thm:main-special}
	    For every integer $\ell \ge 3$ there is a function $g_{\ell}$ such that, if $G$ is a graph with no cycles with exactly $\ell(\ell-2)$ chords, then $\chi(G) \le g_{\ell}(\omega(G))$.
	\end{thm}
	
	The proof of our main result, \Cref{thm:main}, is rather technical. We thus include also a proof of a weaker result, allowing us to find, given large $k$ and a graph whose chromatic number is much larger than its clique number, a cycle with $k'$ chords, where $k' \in \{k, k+1, k+2, k+3\}$; see \Cref{thm:main-almost}.
	This proof already includes a lot of the main ideas that come into the proof of the main, exact result, but avoid many of the technical difficulties.
	
	In the next section, we give an outline of the proofs, focusing on the relatively simpler results \Cref{thm:main-special,thm:main-almost}. After that, in \Cref{sec:prelim}, we define the basic notions that we shall need and mention related observations. In \Cref{sec:bt} we show that, given an integer $k \ge 1$, graphs whose chromatic number is much larger than their clique number either contain a large induced bipartite graph or a cycle with exactly $k$ chords. This will allow us to immediately deduce \Cref{thm:main-special}, and will also be a key component in the proofs of \Cref{thm:main-almost,thm:main}. In \Cref{sec:number-theory} we provide two number theoretic lemmas, related to Lagrange's four square theorem. Finally, in \Cref{sec:approximate} we prove the approximate version of our main result, namely \Cref{thm:main-almost}, and in \Cref{sec:exact} we prove the main result, \Cref{thm:main}.

\section{Proof overview} \label{sec:overview}
    
    The first step in our proof (accomplished in \Cref{sec:bt}) proves the following: given positive integers $\ell$ and~$k$, if $G$ is a graph with large enough chromatic number in terms of a function of its clique number, then it contains either an induced copy of $K_{\ell, \ell}$ or a cycle with exactly $k$ chords (see \Cref{cor:bt-induced-bip}).
    Our proof of this result partly follows the proof of a related result due to Bousquet and Thomass\'e \cite{bousquet-thomasse},
    which asserts that, given $\ell$ and $k$, every graph with large enough chromatic number contains a triangle, an induced $K_{\ell, \ell}$, or a $k$-wheel. 
    Recall that every $(k+2)$-wheel (defined before \Cref{fig:wheel}) contains a cycle with exactly $k$ chords (see \Cref{fig:wheel}). Bousquet--Thomass\'e's result implies a weaker version of our first step, where the graph $G$ is assumed to be triangle-free. To avoid the triangle-free assumption, we leverage the fact that cycles with a prescribed number of chords are easier to find than wheels. This immediately implies our second main result, \Cref{thm:main-special}.
    
    The next step connects to a classical result in number theory, which may be of independent interests.
    Recall that Lagrange's four square theorem asserts that every natural number can be expressed as the sum of four integer squares. 
    In \Cref{sec:number-theory}, we prove two variants of this result. The first (\Cref{lem:twenty-squares}) shows that every large enough integer can be expressed as the sum of exactly 20 integer squares that are all larger than a given constant. The second result, given in \Cref{lem:sum-skewed-squares}, is a similar statement, but here we aim to express a large number $x$ as the sum of integers of the form $a(a+1)$, for large $a$. To achieve this we require that $x$ be divisible by $4$ and we express it as a sum of 80 numbers of the aforementioned form.
    
    In \Cref{sec:approximate}, we prove an approximate version (\Cref{thm:main-almost}) of our main result where instead of finding a cycle with exactly $k$ chords, we find a cycle whose number of chords is in $\{k, k+1, k+2, k+3\}$. To do this, we first find, for some large $\ell$, many induced copies of $K_{\ell, \ell}$ with no edges between them, by applying the first paragraph above several times. To find a cycle with the required number of chords, we take paths of appropriate lengths in the copies of $K_{\ell, \ell}$, and join these by \emph{unimodal paths} (unimodal paths are a commonly used object in the study of $\chi$-boundedness; we will explain what they are in the next section). A key point here is that the unimodal connections contribute $O(\sqrt{k})$ chords (see \Cref{lem:sqrtk-different-Kll,lem:sqrtk-same-Kll}). To get the right number of chords, we apply one of the two number-theoretic lemmas from the previous paragraph, depending on a certain parity condition. The case that requires the use of the second lemma above is the reason why this approach yields only an approximate result.
    
    In \Cref{sec:exact} we prove the exact version of our main result, \Cref{thm:main}. To achieve this we proceed similarly to the paragraph above, but analyse much more carefully the interaction between the various copies of $K_{\ell, \ell}$; this leads to a rather technical proof. We will give an overview of this proof at the beginning of \Cref{sec:exact}.

\section{Preliminaries} \label{sec:prelim}

    Recall that a \emph{chord} in a cycle $C$ is an edge joining two non-consecutive vertices in $C$. 
    Analogously, a \emph{chord} in a path $P$ is an edge that joins two non-consecutive vertices in $P$. 
    It will be convenient to note that, given a graph $G$ and a cycle $C$, the number of chords in $C$ is $e(G[V(C)]) - |C|$.

    We now introduce the notions of \emph{extractions} and \emph{unimodal paths}, which are commonly used in the study of $\chi$-boundedness. The motivation behind these notions is the following observation. For a graph $G$, a vertex $u$ and an integer $i \ge 0$, denote by $N_i(u)$ the set of vertices in $G$ at distance exactly $i$ from $u$. For brevity, given a set of vertices $U$, we write $\chi(U)$ to denote $\chi(G[U])$.
    
    \begin{obs} \label{obs:extraction}
        Let $u$ be a vertex in a graph $G$. Then there is an integer $i \ge 1$ such that $\chi(N_i(u)) \ge \chi(G)/2$. 
    \end{obs}
    
    \begin{proof}
        Let $U_0 = \bigcup_{i \ge 0} N_{2i}(u)$ and let $U_1 = \bigcup_{i \ge 0} N_{2i+1}(u)$. 
        Since there are no edges between distinct $N_{2i}(u)$'s, $\chi(U_0) = \max_{i \ge 0} \chi(N_{2i}(u))$. 
        In particular, there exists $i_0$ such that $\chi(U_0) = \chi(N_{2i_0}(u))$. By the same argument, there exists $i_1$ such that $\chi(U_1) = \chi(N_{2i_1+1}(u))$. Since $\{U_0, U_1\}$ is a partition of $V(G)$, we have $\chi(G) \le \chi(U_0) + \chi(U_1) = \chi(N_{2i_0}(u)) + \chi(N_{2i_1+1}(u))$. Therefore, $\max\{\chi(N_{2i_0}(u), \chi(N_{2i_1+1}(u))\} \ge \chi(G) / 2$, as claimed.
    \end{proof}
    
    Trivially, every vertex in $N_i(u)$ has a `backward' neighbour in $N_{i-1}(u)$.
    Let $N_{\le i}(u) := \bigcup_{0 \le j \le i}N_i(u)$.
    By chasing the backward neighbours of two vertices $x$ and $y$ in $N_i(u)$ repeatedly until they first intersect or have an edge between, we obtain a path between $x$ and $y$ in $N_{\le i}(u)$ that is induced except the edge $xy$, if it is an edge in $G$.
    If this path has length $\ell$, then the distance-$j$ vertices, $j\leq \lfloor\ell/2\rfloor$, from $x$ or $y$ are in $N_{i-j}(u)$.
    
    Together with these paths, \Cref{obs:extraction} can be reformulated in a more systematic way.
    Namely, every graph $G$ has a subgraph $G_1$ with the following properties: 
    \begin{itemize}
        \item 
            $\chi(G_1) \ge \chi(G)/2$,
        \item
            every vertex in $G_1$ has a neighbour in $V(G) \setminus V(G_1)$,
        \item
            for every distinct pair of vertices $x$ and $y$ in $G_1$, there is a path $P$ between $x$ and $y$ satisfying the following: the path is induced except for the edge $xy$; the interior of $P$ lies in $V(G) \setminus V(G_1)$; and the vertices in $P$ apart from $x$, $y$, and their neighbours in $P$ send no edges to $V(G_1)$.
    \end{itemize}
    Such a subgraph $G_1$ is called an \emph{extraction} (or a \emph{$1$-extraction}) of $G$, and the path joining two vertices in $G_1$ with the above properties is called a \emph{unimodal path} with respect to $G_1$. 
    
    One can repeat this process to obtain a sequence $G_0 = G \supseteq G_1 \supseteq \ldots \supseteq G_p$ such that $G_{i}$ is an extraction of $G_{i-1}$. Such a sequence $G_0 \supseteq \ldots \supseteq G_p$ is called \emph{extractions} and $G_p$ is a \emph{$p$-extraction} of $G$. 
    The \emph{$i$-th layer} of an extractions means $V(G_{i-1}) \setminus V(G_{i})$.
    We will often say that $G_p$ is a $p$-extraction of $G$ without specifying the corresponding sequence of extractions.
    Once a $p$-extraction~$G_p$ of $G$ is fixed, a \emph{unimodal path in the $i$-th layer} is a unimodal path with respect to $G_i$ with ends in $G_p$. 
    This is an induced path except for the edges between the ends in $G_p$, whose interior is contained in $V(G_{i-1}) \setminus V(G_i)$. 
    Moreover, the vertices in $P$ apart from the end vertices and their neighbours in the path send no edges to $V(G_i)$. 
    The following observation is immediate from the definitions, yet crucial for our arguments.
    
    \begin{obs} \label{obs:unimodal}
        Let $P$ and $Q$ be unimodal paths in a $p$-extraction of a graph $G$ in layers $i$ and $j$, respectively, with $i < j$.
        Then the vertices in $P$ other than the end vertices and their neighbours in the path send no edges to $V(Q)$.
    \end{obs}
    
    As a final piece of notation, given an index $i \in [0,p]$ and a vertex $x$ in $G_i$, an \emph{$i$-father} of $x$ is a neighbour of $x$ in $V(G_{i-1}) \setminus V(G_i)$.

    We close this section by mentioning the classical K\H{o}vari--S\'os--Tur\'an theorem, which we shall use frequently in what follows.
    The theorem will allow us to find large complete bipartite graphs in bipartite graphs with positive edge density.

    \begin{thm}[K\H{o}v\'ari--S\'os--Tur\'an~\cite{KST54}] \label{thm:kovari-sos-turan}
		For every $\eps>0$ and a positive integer $\ell$, there exists $n_0$ such that the following holds:
		let $G$ be a bipartite graph with a bipartition $A\cup B$ such that $|A|, |B| \ge n_0$. If $G$ has at least $\eps |A| |B|$ edges, then $G$ contains a copy of $K_{\ell, \ell}$.
	\end{thm}

\section{Obtaining large complete bipartite graphs} \label{sec:bt}

	The main purpose of this section is to prove the following theorem, which finds `large' complete bipartite graphs in graphs $G$ with large chromatic number which have no cycle with exactly $k$ chords.
	
    \begin{thm} \label{thm:bt}
		For every integers $k, \ell \ge 1$, there exists a function $g$ such that for every graph $G$ one of the following holds: $\chi(G) \le g(\omega(G))$; $G$ contains a $K_{\ell, \ell}$; or $G$  contains a cycle with exactly $k$ chords.
	\end{thm}
	
	The following corollary of the previous theorem finds induced copies of $K_{\ell,\ell}$ in graphs $G$ with large chromatic number and no cycle with $k$ chords.
	
	\begin{cor} \label{cor:bt-induced-bip}
	    For every integers $k, \ell \ge 1$ there exists a function $h$ such that for every graph $G$ one of the following holds: $\chi(G) \le h(\omega(G))$; $G$ contains an induced $K_{\ell, \ell}$; or $G$ contains a cycle with exactly $k$ chords.
	\end{cor}
	
	\begin{proof}
	    Let $g_{k, \ell}$ be a function as in \Cref{thm:bt} for $k$ and $\ell$. For each $\omega \ge 1$, we define $h_{k, \ell}(\omega) := g_{k, L}(\omega)$ for some $L$ satisfying $L \gg \ell, \omega$. Now consider a graph $G$. By choice of $h_{k, \ell}$, one of the following holds: $\chi(G) \le g_{k, L}(\omega(G)) = h_{k, \ell}(\omega(G))$; $G$ contains a $K_{L, L}$; or $G$ contains a cycle with exactly $k$ chords. In all but the second case, we are done, so suppose that $G$ contains a $K_{L, L}$ and denote the corresponding bipartition by $\{A, B\}$. By Ramsey's theorem and the choice of $L$, each of $G[A]$ and $G[B]$ contain an independent set of size $\ell$. It follows that $G$ contains an induced $K_{\ell, \ell}$, as required. 
	\end{proof}
	
	Notice that this corollary immediately implies \Cref{thm:main-special}, using the observation that an induced $K_{\ell, \ell}$ contains a cycle with exactly $\ell(\ell - 2)$ chords.
	
	Before turning to the proof of \Cref{thm:bt}, we mention two preliminary results.
	The first one, due to Thomas and Wollan~\cite{thomas-wollan}, shows that $10k$-connected graphs are \emph{$k$-linked}, namely one can join any $k$ pairs of vertices by pairwise vertex-disjoint paths. This improved on an earlier result of Bollob\'as and Thomason~\cite{bollobas1996highly}.

    \begin{thm}[Thomas--Wollan~\cite{thomas-wollan}] \label{thm:linkage}
		Let $k \ge 0$ and let $G$ be a $10k$-connected graph. For every set of (not necessarily distinct) vertices $x_1, \ldots, x_k, y_1, \ldots, y_k$ there exist paths $Q_1, \ldots, Q_k$ with pairwise disjoint interiors, such that $Q_i$ has ends $x_i, y_i$.
	\end{thm}

    The second preliminary result, due to K\"uhn and Osthus \cite{kuhn2004induced}, allows us to find a subdivision of a $\kappa$-connected graph.
    
	\begin{thm}[K\"uhn--Osthus~\cite{kuhn2004induced}] \label{thm:ko}
	    Let $\chi \gg \kappa, \ell$. 
	    Then for every graph $G$ one of the following holds: $\chi(G) \le \chi$; $G$ contains a $K_{\ell, \ell}$; or $G$ contains, as an induced subgraph, a $1$-subdivision of a $\kappa$-connected graph $H$.
	\end{thm}
	
	Recall that a $1$-subdivision of a graph $H$ is the graph obtained from $H$ by replacing each edge $uv$ by a path $uw_{uv}v$, where the vertices $w_{uv}$ are distinct. 
	For brevity, denote by $H_\sbb$ the 1-subdivision of a graph $H$. 
  
	\subsection{Finding a large complete bipartite graph}
	
	In the remainder of this section, we prove \Cref{thm:bt}.
	To do so, we assume that we are given a graph $G$ with large chromatic number and which contains no large complete bipartite subgraphs. The starting is \Cref{thm:ko} that says that under this assumption $G$ contains an induced copy of a $1$-subdivision of a highly connected graph $H$. In fact, we may assume that this subdivision lies in a $p$-extraction of $G$, for some large $p$. Our proof splits-off into two cases: either we can emulate the triangle-freeness assumption that Bousquet and Thomass\'e \cite{bousquet-thomasse} impose (when proving that every triangle-free graph with large chromatic number contains either a large complete bipartite subgraph or a long wheel), or not. In the former case our proof largely follows \cite{bousquet-thomasse}, and in the latter case we use the abundance of triangles to construct a cycle with the right number of chords.
	
	Throughout this section, we use the following setup. The parameters are chosen according to the hierarchy
    \begin{equation*}
        \chi \gg p \gg \kappa \gg k_1 \gg k_2 \gg k, \ell, \omega. 
    \end{equation*}
    Let $G$ be a $K_{\ell,\ell}$-free graph with $\chi(G) = \chi$ and $\omega(G) \le \omega$.
    Let $G \supseteq G_1 \supseteq \ldots \supseteq G_p$ be a sequence of extractions, so that $\chi(G_p) \ge 2^{-p}\chi$. We will show that $G$ contains a cycle with exactly $k$ chords.

    By \Cref{thm:ko}, there is a $\kappa$-connected graph $H$ whose $1$-subdivision $H_\sbb$ is an induced subgraph of~$G_p$. Fix an induced copy of $H_\sbb$ in $G_p$ and denote by $\sub(e)$ the vertex of the induced copy of $H_\sbb$ in $G_p$ that subdivides the edge $e$ in $H$. We also identify vertices in $H$ as corresponding vertices in the copy of $H_\sbb$ (and hence, in $G_p$). When considering a neighbour $z$ of a vertex $x$ in $H$, we say that $z$ is an $H$-neighbour of $x$ to stress that the adjacency is not in the host graph $G_p$.
    
    The following lemma allows us to deal with the case where at every vertex $x$ in $H$, for many $H$-neighbours $z$ of $x$, $\sub(xz)$ has a common $j$-father with at least one of $x$ and $z$. 

	\begin{lem} \label{lem:all-x-bad}
		Suppose that every vertex $x$ in $H$ is incident with at least $2k_1$ edges $e = xy$ in $H$ such that $\sub(e)$ has a common $j$-father with at least one of $x$ and $y$, for at least $k_1$ indices $j$.
		Then $G$ contains a cycle with exactly $k$ chords.
	\end{lem}  
	We postpone the proof of this lemma until the end of this section.
	In the light of \cref{lem:all-x-bad}, we may assume that the particular structure described therein does not appear in the fixed copy of $H_\sbb$ in~$G_p$.
	That is, there is a `special' vertex $x$ in $H$ such that for all but at most $2k_1$ neighbours $z$ of $x$ in $H$ the following holds: for all but at most $k_1$ values of $j$, the vertex $\sub(xz)$ does not have a common $j$-father with neither $x$ nor $z$. 
	Let $Z$ be a set of $k_1$ neighbours of $x$ in $H$ that satisfy this property, i.e., for $z\in Z$, there are at most $2k_1$ indices $j$ such that $\sub(xz)$ has a common $j$-father with $x$ or $z$.
	
	Then all but at most $2k_1^2$ values of $j$ satisfy: for every $z \in Z$, the vertex $\sub(xz)$ does not have a common $j$-father with neither $x$ nor $z$.
	Let $J$ be the set of `good' indices $j$ that satisfy this property. In particular, $|J| \ge p - 2k_1^2$. The vertex $x$, set $Z \subseteq N_H(x)$ and set $J \subseteq [p]$ will be fixed throughout this section. Let $Y$ be the set of vertices in the fixed copy of $H_\sbb$ in $G_p$ that are adjacent to $x$ and some vertex in~$Z$, i.e., $Y := \{\sub(xz) : z \in Z\}$. 
	For $y\in Y$, the \emph{extended neighbour} $\ext(y)$ of $y$ is the vertex $z\in Z$ such that $y=\sub(xz)$.
	In particular, $|Z|=|Y|=k_1$.
    
    Following the definition in page 7 of \cite{bousquet-thomasse}, a collection of unimodal paths $\Q$ with the set $Y' \subseteq Y$ of endpoints is \emph{$\alpha$-good} for $\alpha>0$ if the following conditions hold:
    \begin{enumerate} [label = \rm(G\arabic*)]
        \item
            For every $y \in Y'$, there exists a unique path $Q\in\Q$ with the endpoint $y$ and moreover, $Q$ is the only path that contains an edge incident with $y$.
        \item
            For every $y \in Y'$, no vertex in any of the paths in $\Q$ contains a neighbour of the extended neighbour $\ext(y)$ of $y$ other than $y$ itself.
        \item
            For any distinct $Q, Q' \in \Q$, there is no edge between a father (in $Q$) of an endpoint in $Q$ and a father (in $Q'$) of an endpoint in $Q'$.
        \item
            For every $y \in Y'$ and every $Q \in \Q$, there are at most $\alpha|N_{H_\sbb}(\ext(y))|$ vertices in $N_{H_\sbb}(\ext(y))$ that are adjacent to a vertex in $Q$.
    \end{enumerate}

    We say that a collection of unimodal paths $\Q$ 
    is \emph{independent} if there are no edges between distinct paths in $\Q$.    
	The following is a variant of the first part of (the proof of) Lemma 9 in \cite{bousquet-thomasse}.
	\begin{lem} \label{lem:good-x}
	    There is a $\frac{1}{4k_2}$-good collection of unimodal paths $\Q$ of size at least $k_2$ in $G$.
	\end{lem}
    \begin{proof}
			Let $k', k'', c$ be such that $k_1 \gg k' \gg k'' \gg c \gg k_2$.
            For each $y \in Y$ and $j \in J$, choose a $j$-father of $y$ and denote it by $f_j(y)$. 
            We wish to find `large' subsets $Y' \subseteq Y$ and $J' \subseteq J$ such that, for every fixed $j$, the fathers $f_j(y)$ are distinct for $y \in Y'$.
            To this end, we construct sequences $Y = Y_0 \supseteq \ldots \supseteq Y_\ell$ and $j_1, \ldots, j_\ell \in J$ recursively as follows: 
            given $Y_0, \ldots, Y_t$ and $j_1, \ldots, j_t$ for $t < \ell$, if there exists $j \in J \setminus \{j_1, \ldots, j_t\}$ and a subset $Y' \subseteq Y_t$ of size at least $\sqrt{|Y_t|}$ such that $f_j(y)$ is the same for all $y \in Y'$, define $j_{t+1} = j$ and $Y_{t+1} = Y'$. Otherwise, stop the process.
            
			Suppose first that $j_1, \ldots, j_\ell$ and $Y_\ell$ are well-defined through the recursive process. Then $f_j(y)$ is the same for all $y \in Y_{\ell}$, for every $j \in \{j_1, \ldots, j_{\ell}\}$. Then $\{f_j(y) : j \in \{j_1, \ldots, j_{\ell}\}\}$ is a set of size $\ell$ that is fully joined to $Y_{\ell}$. 
			Since $|Y_{\ell}| \ge |Y|^{2^{-\ell}}=(k_1)^{2^{-\ell}} \ge \ell$, there is a copy of $K_{\ell, \ell}$ in $G$, contradicting the assumption that $G$ is $K_{\ell, \ell}$-free.
            
			Now suppose that the process stops at the $t$-th step for some $t<\ell$ with the outputs $j_1, \ldots, j_t$ and~$Y_t$. 
			Let $J' := J \setminus \{j_1, \ldots, j_t\}$.
			Then for every $j \in J'$, each element in the multiset $\{f_j(y) : y \in Y_t\}$ repeats at most $\sqrt{|Y_t|}$ times. 
			Since $|Y_t| \ge |Y|^{2^{-\ell}}=(k_1)^{2^{-\ell}} \ge (k'')^2$, for every $j \in J'$ there is a subset $Y_j' \subseteq Y_t$ of size $k''$ such that $f_j(y)$ are distinct for all $y \in Y_j'$. 
			Let $Y' \subseteq Y_t$ be the most popular choice for $Y_j'$. By averaging, $Y_j' = Y'$ for at least $|J'| / \binom{|Y_t|}{\sqrt{|Y_t|}} \ge (p - k_1^2 - \ell)/2^{k_1} \ge k'$ indices $j \in J'$; let $J'' \subseteq J'$ be a set of size $k'$ such that $Y_j' = Y'$ for every $j \in J''$.

			Let $W := \{f_j(y) : y \in Y', j \in J''\}$. We claim that for every $y \in Y'$ there are at most $k''$ elements $w \in W$ such that $w$ sends at least $\frac{1}{4k_2}|N_{H_\sbb}(\ext(y))|$ edges to $N_{H_\sbb}(\ext(y))$. 
			Suppose to the contrary that this is not the case for $y \in Y'$. Write $N := N_{H_\sbb}(\ext(y))$ for brevity and let $W'$ be the set of vertices $w\in W$ such that $w$ sends at least $\frac{1}{4k_2}|N|$ edges to $N$ satisfying $|W'| \ge k''$. 
			By \Cref{thm:kovari-sos-turan}, $G[N, W']$\footnote{Here $G[A,B]$ denotes the \emph{bipartite} induced subgraph, i.e., we only take those edges that cross between $A$ and $B$.} contains a copy of $K_{\ell, \ell}$, a contradiction. 
			Consider the set 
			\begin{equation*}
				\left\{j \in J'' : \text{ for all $z,y\in Y'$, $f_j(z)$ sends at most $\frac{1}{4k_2}|N_{H_\sbb}(\ext(y))|$ edges to $N_{H_\sbb}(\ext(y))$}\right\}.
			\end{equation*}
			By the above statement, this set has size at least $|J''| - |Y'| \cdot k'' \ge k' - (k'')^2 \ge k_2$. Let $J_{\final}$ be a subset of the above set of size $k_2$.

			For every $j \in J_{\final}$, define a graph $F_j$ on vertices $Y'$ whose edges are pairs $y_1 y_2$ where $y_1, y_2 \in Y'$ are distinct and one of the following holds: $f_j(y_i)$ is adjacent to $y_{3-i}$ for some $i \in [2]$; $f_j(y_i)$ is adjacent to $\ext(y_{3-i})$; or $f_j(y_1)$ and $f_j(y_2)$ are adjacent. By Ramsey's theorem, the graph $F = \bigcup_{j \in J_{\final}}F_j$ contains either an independent set of size $2k_2$ or a clique of size $c$. 
			Assume the latter case and suppose that $U$ is a subset of $Y'$ of size $c$ that induces a clique. Then for some $j \in J_{\final}$ the graph $F_j$ has at least $\frac{1}{k_2}\binom{|U|}{2}$ edges. Let $U' := U \cup \{\ext(y) : y \in U\} \cup \{f_j(y) : y \in U\}$. Then, by definition of $F_j$, 
			\begin{equation*}
				e(G[U']) \ge e(F_j) \ge \frac{1}{k_2}\binom{|U|}{2} \ge \frac{1}{k_2} \binom{|U'|/3}{2}.
			\end{equation*}
			By \Cref{thm:kovari-sos-turan}, $G[U']$ contains a copy of $K_{\ell, \ell}$, a contradiction.

			It remains to consider the case where $F$ contains an independent set $U$ of size $2k_2$. Write $U = \{y_1, \ldots, y_{2k_2}\}$ and $J_{\final} = \{j_1, \ldots, j_{k_2}\}$. Let $Q_i$ be a unimodal path between $f_{j_i}(y_{2i-1})$ and $f_{j_i}(y_{2i})$, for $i \in [k_2]$. By construction, one may  easily check that $\{Q_1, \ldots, Q_{k_2}\}$ is a $\frac{1}{4k_2}$-good collection of size at least $k_2$.
        \end{proof}

	Recall that a \emph{$k$-wheel} in a graph $F$ is an induced cycle $C$ along with an additional vertex that has at least $k$ neighbours in $C$. We note that a $(k+2)$-wheel contains a cycle with exactly $k$ chords. Indeed, suppose that $C$ is an induced cycle and $v$ is a vertex (not in $C$) with at least $k+2$ neighbours in $C$. Let $u_0, \ldots, u_{k+1}$ be $k+2$ consecutive neighbours of $u$ in $C$. Let $P$ be the subpath in $C$ that starts at $u_0$, ends in $u_{k+1}$ and contains the vertices $u_1, \ldots, u_k$, and let $C'$ be the cycle obtained by concatenating $P$ and the path $u_{k+1} v u_0$. Then $C'$ is a cycle with exactly $k$ chords; its chords are $vu_1, \ldots, vu_k$. 
	
	The following lemma obtains a collection of independent good paths from a collection of good paths. We use this as a black box.
	\begin{lem}[Lemma 17 in \cite{bousquet-thomasse}] \label{lem:make-Q-independent}
        Let $\Q$ be an $\alpha$-good collection of unimodal paths of size $k_2$.
        If $G$ is $k$-wheel-free, then
        there exists an independent $2\alpha$-good collection of unimodal paths of size at least~$k$.
	\end{lem}
	
	The following lemma, together with the lemmas above, will complete the proof of \Cref{thm:bt}.
	
	\begin{lem}[Lemma 12 in \cite{bousquet-thomasse}] \label{lem:find-k-wheel}
		If there is a collection of $\ceil{k/2}$ independent $(1/2k)$-good unimodal paths with endpoints in $N_{H_\sbb}(x)$ for a vertex $x$ in $H$, then the graph $G$ has a $k$-wheel.
	\end{lem}
	\begin{proof}[Proof of~\cref{thm:bt}]
	 Recall that, while assuming~\Cref{lem:all-x-bad},  the parameters are chosen according to the hierarchy $\chi \gg p \gg \kappa \gg k_1 \gg k_2 \gg k, \ell, \omega$.
	 By \cref{lem:good-x}, there is a collection $\Q$ that consists of $k_2$ paths that are $\frac{1}{4k_2}$-good.
	 Then by~\cref{lem:make-Q-independent}, there is a collection of $k_2$ independent $\frac{1}{2k_2}$-good paths; however,~\cref{lem:find-k-wheel} then finds a $k_2$-wheel in $G$, and hence, a cycle with $k$ chords.
	\end{proof}

    \subsection{Dealing with triangles} \label{subsec:triangles}
    
    It remains to prove~\cref{lem:all-x-bad}. The following lemma will be useful in the proof.
		
		\begin{lem} \label{claim:neighbours-Hs}
			Suppose that $G$ does not have a cycle with exactly $k$ chords. 
			Then for every vertex $u$ in $G$, there exists a subset $R\subseteq V(H)$ of at most $3(k+1)$ vertices such that~$u$ does not have any edges to the copy of $(H \setminus R)_\sbb$ in $G$.
		\end{lem}

		\begin{proof}
			Let $X$ be the set of vertices in $H$ that are neighbours of $u$ in $G$. We claim that $|X| \le k+1$. Suppose to the contrary that $x_1, \ldots, x_{k+2}$ are distinct vertices in $X$. 
			As $H$ is $\kappa$-connected with $\kappa \gg k$, \Cref{thm:linkage} guarantees that there exist pairwise internally vertex-disjoint paths $Q_1, \ldots, Q_{k+2}$ such that each $Q_i$ ends at $x_i$ and $x_{i+1}$, where the addition of indices taken modulo $k+2$.
			Let $Q_i'$ be the path in the copy of $H_\sbb$ in $G$ that corresponds to the 1-subdivision of $Q_i$.
			Then by concatenating the paths $Q_1', \ldots , Q_{k+2}'$, we obtain an induced cycle in $G$ that contains at least $k+2$ neighbours of $u$. 
			Then there exists a cycle with exactly $k$ chords, which contradicts the assumption.

			Let $M$ be a maximum matching in $H$ such that, for each edge $e$ in $M$, $\sub(e)$ is adjacent to $u$ in~$G$. 
			Our next claim is $|M| \le k+1$. Suppose to the contrary that $x_1 y_1, \ldots, x_{k+2} y_{k+2}$ are vertex-disjoint edges in $M$. 
			By using~\cref{thm:linkage} as before, there exist vertex-disjoint paths $Q_1, \ldots, Q_k$ in $H \setminus \{x_1 y_1, \ldots, x_{k+2} y_{k+2}\}$ such that each $Q_i$ ends at $y_i$ and $x_{i+1}$.
			Again, let $Q_i'$ be the 1-subdivision of $Q_i$ in the copy of $H_\sbb$ in $G$.
			Then $x_1 y_1 Q_1' \ldots x_{k+2} y_{k+2} Q_{k+2}'x_1$ is an induced cycle in $G$ that contains at least $k+2$ neighbours of $u$, which again yields a contradiction.

			Take $R = X \cup V(M)$. Then $|R| \le 3(k+1)$ and $u$ has no neighbours in $(H \setminus R)_\sbb$, as required.
		\end{proof}

		\begin{proof}[Proof of \Cref{lem:all-x-bad}]
		    Let $k_3, k_4$ be such that 
			Let $k_1 \gg k_2 \gg k_3 \gg k_4 \gg \omega \gg k$.
			Say that an edge $e = xy$ of $H$ is \emph{triangulated} if $\sub(e)$ and at least one of $x$ and $y$ have a common $j$-father, for at least $k_1$ indices $j$.
			Then every vertex in $H$ is incident with at least $2k_1$ \triangulated{} edges and hence, there is a matching $M_1$ in~$H$ that consists of $k_1$ \triangulated{} edges.
			
			For each $e \in M_1$ let $x(e)$ and $y(e)$ be the two ends of $e$, let $f(e)$ be a $j(e)$-father of $\sub(e)$ which is adjacent to $x(e)$ or $y(e)$, so that the indices $j(e)$ are all distinct for $e \in M_1$. Let $U$ be the set $\{f(e) : e \in M_1\}$; then $|U|=k_1$. Since $k_1 \gg k_2, \omega$ and $G$ is $K_{\omega+1}$-free, the set $U$ contains an independent set of size $k_2$. Let $M_2$ be a submatching of $M_1$ of size $k_2$ such that $\{f(e) : e \in M_2\}$ is independent. 

			Let $F$ be the auxiliary graph whose vertices are the edges of $M_2$, where $e_1 e_2$ is an edge in $F$ whenever $f(e_{3-i})$ is adjacent to at least one of $\sub(e_i), x(e_i)$ and $y(e_i)$ for some $i \in [2]$. 
			By \Cref{claim:neighbours-Hs}, the vertex $\sub(e)$ is adjacent to at most $3(k+1)$ vertices in the copy of $H_\sbb$ in $G$, for every $e \in M_2$. 
			In particular, every $e\in M_2$ has neighbours in at most $3(k+1)$ of the sets $\{\sub(e'), x(e'), y(e')\}$ with $e' \in M_2$.
			Thus, $F$ has at most $3(k+1) |F|$ edges. Tur\'an's theorem then implies that $F$ contains an independent set of size at least $|F|/6(k+1) \ge k_3$. Let $M_3$ be a submatching of $M_2$ of size $k_3$ that forms an independent set in $F$.
			That is, for every distinct $e, e'\in M_3$, $f(e)$ is \emph{not} adjacent to any $\sub(e')$, $x(e')$ or $y(e')$.

			For each $e \in M_3$, let $R(e)$ be a set of vertices in $H$ of size at most $3(k+1)$ such that $x(e'), y(e') \notin R(e)$ for all $e' \in M_3$ and $f(e)$ has no neighbours in $(H \setminus R(e))_\sbb$ except for possibly $\sub(e)$, $x(e)$ and $y(e)$; the existence of such a set is guaranteed by the choice of $M_3$ and by \Cref{claim:neighbours-Hs}.

			We consider three cases. Throughout the case analysis, we abuse notation by writing $e$ for the vertex $\sub(e)$ for simplicity.
			Whenever we consider indexed edges $e_1,e_2,\ldots$, we shall denote $x_i := x(e_i)$, $y_i := y(e_i)$, $f_i := f(e_i)$ and $R_i := R(e_i)$.
			The 1-subdivision of a path $Q_i$ in $H$ is denoted by~$Q_i'$, and is again a path in the copy of $H_\sbb$.
			
			\noindent\textbf{Case 1:} there are $k$ edges $e\in M_3$ such that $f(e)$ is not adjacent to  $y(e)$.

				Let $e_1, \ldots, e_k \in M_3$ be distinct edges such that $f_i$ is not adjacent to $y_i$ for $i \in [k]$ (by choice of $M_3$ this means that $f_i$ is adjacent to $x_i$ for $i \in [k]$). 
				Let $R := R_1 \cup \ldots \cup R_k$. As $H$ is $\kappa$-connected, $H \setminus R$ is $10k$-connected and thus, by \Cref{thm:linkage}, there exist pairwise vertex-disjoint paths $Q_1, \ldots, Q_k$ in $H \setminus R$ such that $Q_i$ has ends $y_i$ and $x_{i+1}$ for $i \in [k]$, where the addition of indices is taken modulo $k$.
				Let $C$ be the cycle $x_1 f_1 e_1 y_1 Q_1' \ldots x_k f_k e_k Q_k'x_1$. Then~$C$ has exactly $k$ chords, namely: $x_1 e_1$, \ldots, $x_k e_k$.

				\begin{figure}[ht]
					\centering
					\includegraphics[scale = 1]{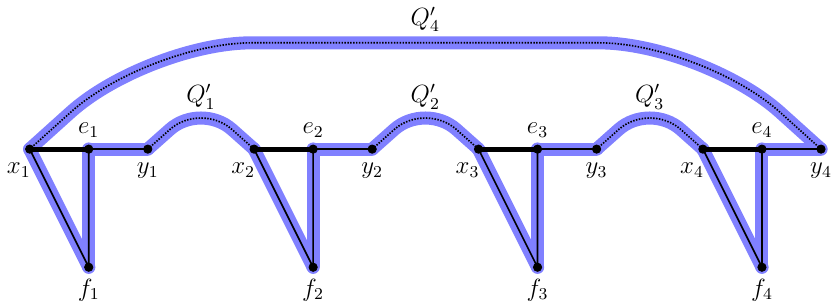}
					\vspace{-.4cm}
					\caption{Case 1: $f_i$ not adjacent to $y_i$}
					\label{fig:case-1}
					\vspace{-.2cm}
				\end{figure}
            
            \medskip
			\noindent\textbf{Case 2:} $k$ is even and $f(e)$ is adjacent to $x(e)$ and $y(e)$ for at least $k/2$ values of $e$ in $M_3$.

				Write $k = 2s$.
				Let $e_1, \ldots, e_{s} \in M_3$ be distinct edges such that $f_i$ is adjacent to $x_i$ and $y_i$ for $i \in [s]$. Let $R := R_1 \cup \ldots \cup R_s$. 
				As before, there exist pairwise vertex-disjoint paths $Q_1, \ldots, Q_{s}$ in $H \setminus R$ such that $Q_i$ has ends $y_i$ and $x_{i+1}$ for $i \in [s]$, where the addition of indices is taken modulo $s$. 
				Then $x_1 f_1 e_1 y_1 Q_1' \ldots x_{s} f_{s} e_{s} y_{s} Q_{s}'x_1$ is a cycle in $G$ with exactly $k$ chords: $x_1 e_1, \ldots, x_{s}e_{s}$ and $f_1 y_1, \ldots, f_{s} e_{s}$. 

				\begin{figure}[ht]
					\centering
					\includegraphics[scale = 1]{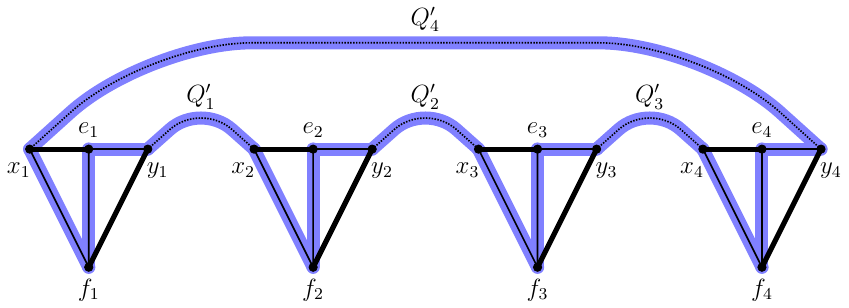}
					\vspace{-.4cm}
					\caption{Case 2: $k$ even $f_i$ adjacent to $x_i$ and $y_i$}
					\label{fig:case-2}
					\vspace{-.2cm}
				\end{figure}
            
            \medskip
			\noindent\textbf{Case 3:} $k$ is odd and $f(e)$ is adjacent to $x(e)$ and $y(e)$ for at least $k_4 + 1$ values of $e$ in $M_3$.

				Note that if $k = 1$ then $f(e)\, x(e)\, e\, y(e)$ is a cycle with exactly one chord, for any $e \in M_3$ where $f(e)$ is adjacent to $x(e)$ and $y(e)$. We may thus assume that $k \ge 3$; write $k = 2s- 1$ with $s>1$.

				Let $M_4$ be a submatching of $M_3$ of size $k_4$ that consists of edges $e$ where $f(e)$ is adjacent to $y(e)$ and $j(e) > 1$ (recall that $M_3$ contains at most one edge $e$ with $j(e) = 1$). 
				For each $e \in M_4$, let $g(e)$ be a $1$-father of $f(e)$. Here the $g(e)$'s are not necessarily distinct. We claim that each $g(e)$, with $e\in M_4$, is adjacent to at most $k+1$ of the vertices $f(e')$ with $e' \in M_4$. 
				Indeed, suppose that $e_1, \ldots, e_{k+2}$ are distinct edges in $M_4$ such that each $f_j$, with $j \in [s]$, is adjacent to $g(e)$. As usual, we can find paths $Q_1, \ldots, Q_{k+2}$ in $H \setminus (R(e_1) \cup \ldots \cup R(e_{k+2}))$ that are pairwise vertex-disjoint and $Q_i$ has ends $y_i$ and $x_{i+1}$ for $i \in [k+2]$. 
				Then $x_1 f_1 y_1 Q_1' \ldots x_{k+2} f_{k+2} y_{k+2} Q_{k+2}'x_1$ is an induced cycle in $G$ that contains at least $k+2$ neighbours of $g(e)$.
				Thus, a cycle with exactly~$k$ chords exists. 

				Let $F$ be the auxiliary graph on the edges in $M_4$, where $e_1$ and $e_2$ form an edge if $g(e_{3-i})$ is adjacent to at least one of $e_i, x(e_i), y(e_i),$ and $f(e_i)$ for some $i \in [2]$. 
				By \Cref{claim:neighbours-Hs} and the previous paragraph, each $g(e)$, with $e\in M_4$, is adjacent to at most $4k+4$ vertices amongst $\{e', x(e'), y(e'), f(e'): e' \in M_4\}$. 
				Therefore,~$F$ has at most $(4k+4)|F|$ edges and thus, it has an independent set of size at least $|F|/(8k+4) \ge k_4/(8k+4) \ge \ceil{k/2}$. Let $e_1, \ldots, e_{s}$ be distinct edges in $M_1$ that form an independent set in $F$. 
				Write $g_i:=g(e_i)$ and let $R_i'$ be a set of at most $3(k+1)$ vertices in $H$ such that $g_i$ has no neighbours in $(H \setminus R_i')_\sbb$ except for possibly $x_i, y_i, e_i$, and $x_j, y_j, e_j \notin R_i'$ for $j \in [s]$. 
				Such a set $R_i'$ exists by \Cref{claim:neighbours-Hs} and the fact that $e_i$'s form an independent set in $F$. 
				Observe that the $g_i$'s are distinct (since $g_i$ is adjacent to $f_i$ but not to $f_j$ with $j \in [s] \setminus \{i\}$), and $g_i$ sends no edges to $\{x_j, y_j, e_j, f_j\}$ for $j \neq i$.
				Let $R = R_1 \cup \ldots \cup R_{s} \cup R_1' \cup \ldots \cup R_{s}'$.
				As usual, let $Q_1, \ldots, Q_{s-1}$ be pairwise vertex-disjoint paths in $H \setminus R$ such that each $Q_i$ ends at $y_i$ and~$x_{i+1}$. 
				Then the path $y_1 Q_1' x_2 f_2 e_2 y_2 Q_2' \ldots x_{s-1} f_{s-1} e_{s-1} y_{s-1} Q_{s-1}x_s$ has exactly $2(s-2)=k-3$ chords. Similarly, paths with exactly $k-3$ chords exist between any $u\in\{x_1, y_1\}$ and $v\in \{x_{s}, y_{s}\}$.
				To extend one such path to a cycle with exactly $k$ chords, we need the following claim.
				
				\begin{claim} \label{claim:paths-one-fathers}
					There exist paths $P_1$ and $P_{s}$ such that $P_i$ has ends $g_i$ and one of $x_i, y_i$ and $V(P_i) \subseteq \{x_i, y_i, e_i, f_i, g_i\}$, for $i \in \{1, s\}$, and $P_1$ and $P_{s}$ have three chords in total.
				\end{claim}
				\begin{proof}[Proof of the claim]
					Write $x = x_i$, $y = y_i$, $e = e_i$, $f = f_i$, $g = g_i$ for some $i \in \{1, s\}$, and let $\sigma$ be the number of neighbours of $g$ in $\{x, y, e\}$.
					We aim to show that there are paths $P$ with vertices in $\{x, y, e, f, g\}$ whose ends are $g$ and one of $x, y$ satisfying the following requirements (separately, namely $P$ varies).  
					\begin{enumerate}[label = \rm(\arabic*)] 
					    \setcounter{enumi}{-1}
						\item \label{itm:zero}
							$P$ has no chords, if $\sigma = 1$;
						\item \label{itm:one}
							$P$ has exactly one chord, if $\sigma \neq 1$;
						\item \label{itm:two}
							$P$ has exactly two chords;
						\item \label{itm:three}
							$P$ has three chords, if $\sigma = 1$.
					\end{enumerate}
					By using this, 		\Cref{claim:paths-one-fathers} easily follows. Indeed, unless $\sigma = 1$ for both $i = 1$ and $s$, we can take one of $P_1$ and $P_{s}$ to have one and two chords, respectively, using \ref{itm:two} and \ref{itm:one}. 
					Otherwise, we can take one to have no chords and the other to have three, using \ref{itm:zero} and \ref{itm:three}.
					To show that paths $P$ that satisfies (1)--(4) exist, we consider the four possible values of $\sigma$.

					\begin{figure}[ht]
				        \def \scale {.9}	
						\centering
						\begin{subfigure}[b]{.16\textwidth}
				            \centering
				            \includegraphics[scale = \scale]{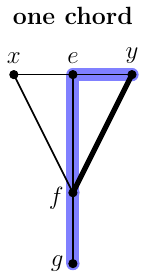}
                            
                            \vspace{.7cm}
				            \includegraphics[scale = \scale]{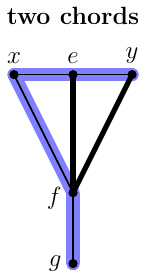}
				            \vspace{.6cm}
				            \caption*{\Large $\sigma = 0$}
						\end{subfigure}
						\hspace{.1cm}
						\begin{subfigure}[b]{.35\textwidth}
				            \centering
				            \includegraphics[scale = \scale]{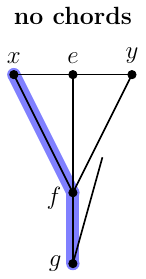}
				            \hspace{.2cm}
				            \includegraphics[scale = \scale]{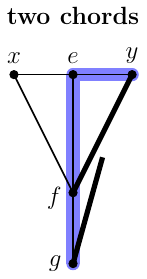}
				            
				            \vspace{.7cm}
				            
				            \includegraphics[scale = \scale]{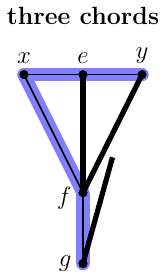}
				            \vspace{.6cm}
				            \caption*{\Large $\sigma = 1$}
						\end{subfigure}
						\hspace{.1cm}
						\begin{subfigure}[b]{.16\textwidth}
				            \centering
				            \includegraphics[scale = \scale]{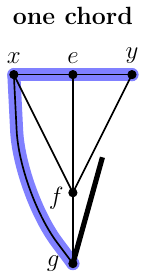}
                            
                            \vspace{.7cm}
				            \includegraphics[scale = \scale]{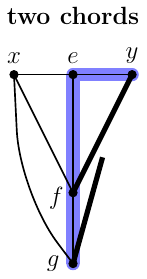}
				            \vspace{.6cm}
				            \caption*{\Large $\sigma = 2$}
						\end{subfigure}
						\hspace{.2cm}
						\begin{subfigure}[b]{.16\textwidth}
				            \centering
				            \includegraphics[scale = \scale]{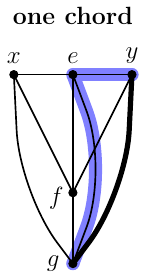}
                            
                            \vspace{.7cm}
				            \includegraphics[scale = \scale]{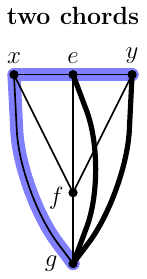}
				            \vspace{.6cm}
				            \caption*{\Large $\sigma = 3$}
						\end{subfigure}
						\caption{Proof of \Cref{claim:paths-one-fathers}}
						\label{fig:paths-one-fathers}
					\end{figure}

					\begin{itemize}
						\item
							$\sigma = 0$. 
							 For \ref{itm:one} take $P = gfey$, and for \ref{itm:two} take $P = gfxey$.
						\item
							$\sigma = 1$. Without loss of generality, $g$ is not adjacent to $x$.
							For \ref{itm:zero}, \ref{itm:two}, and \ref{itm:three}, take $P = gfx$, $P = gfey$, and $P = gfxey$, respectively.
						\item
							$\sigma = 2$. Without loss of generality, $g$ is adjacent to $x$. 
							For \ref{itm:one} and \ref{itm:two}, take $P = gxey$ and $P = gfey$, respectively.
						\item
							$\sigma = 3$. 
							For \ref{itm:one} and \ref{itm:two}, take $P = gey$ and $P = gxey$, respectively. \qedhere
					\end{itemize}
				\end{proof}

				Let $P_1$ and $P_{s}$ be as in~\cref{claim:paths-one-fathers}. Without loss of generality, we may assume that $y_1$ is an end of $P_1$ and $x_{s}$ is an end of $P_{s}$. Let $P$ be a unimodal path with ends $g_1$ and $g_{s}$. 
				Then the interior of $P$ sends no edges to the copy of $H_\sbb$ or $\{f_1, \ldots f_{s}\}$, since each $g_i$ is a $1$-father of $f_i$'s, whereas
				the $f_i$'s and the copy of $H_\sbb$ are in $G_2$. 
				Finally, augmenting the path $y_1 Q_1' x_2 f_2 e_2 y_2 Q_2' \ldots x_{s-1} f_{s-1} e_{s-1} y_{s-1} Q_{s-1}'x_s$ by adding the path $x_s P_{s}g_s P g_1 P_1 y_1$ gives a cycle with exactly $k$ chords.
				\begin{figure}[ht]
					\centering
					\includegraphics[scale = 1]{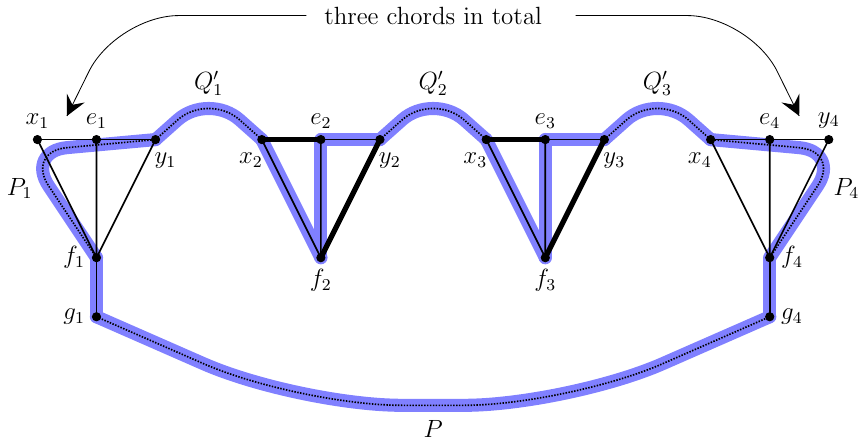}
					\vspace{-.4cm}
					\caption{Case 3: odd $k$ and $f_i$ adjacent to $x_i$ and $y_i$}
					\label{fig:3}
				\end{figure}
		\end{proof}

\section{Number theory lemmas} \label{sec:number-theory}
    
    In this section we prove two results which are variants Lagrange's four-square theorem, which asserts that every positive integer can be written as the sum of at most four integer squares.

	\begin{lem} \label{lem:twenty-squares}
		For every $c$, every large enough $k$ can be written as a sum of exactly $20$ squares larger than $c^2$.
	\end{lem}

	\begin{proof}
	    Given a non-negative integer $x$, let $f(x)= x^2 + 2500c^2 + (4c+1)^2$. We first claim that for every integer $x\geq 0$, the number $f(x)$ can be written as the sum of exactly five squares larger than~$c^2$. If $x \ge c$, this follows by writing
	    \begin{equation*}
	        f(x) = x^2 + (30c)^2 + (24c)^2 + (32c)^2 + (4c+1)^2.
	    \end{equation*}
	    If $x \le c$ and $x$ is even, we have
	    \begin{equation*}
	        f(x) = 2\left(\frac{50c + x}{2}\right)^2 + 2\left(\frac{50c - x}{2}\right)^2+(4c+1)^2,
	    \end{equation*}
	    which readily implies that $f(x)$ is the sum of five squares larger than $c^2$. 
	    Finally, if $x \le c$ and $x$ is odd, we use the following equality to reach the same conclusion.
	    \begin{equation*}
	        x^2 + (4c + 1)^2 
	        = 2\left(\frac{4c + 1 + x}{2}\right)^2 + 2\left(\frac{4c + 1 - x}{2}\right)^2.
	    \end{equation*}

        Let $\ell := k - 4 \left( 2500c^2 + (4c+1)^2 \right)$ and suppose that $k$ is large enough so that $\ell \ge 0$.
		By the four squares theorem, there exist non-negative integers $x_1, \ldots, x_4$ such that $\ell = x_1^2 + x_2^2+x_3^2 + x_4^2$. Equivalently,
		\begin{equation*}
			k = \sum_{i =1}^4 \left(x_i^2 + 2500c^2 + (4c+1)^2\right) = \sum_{i =1}^4 f(x_i).
		\end{equation*}
		By using the fact that each $f(x_i)$ is the sum of five squares larger than $c^2$, we conclude that $k$ is the sum of $20$ squares larger than $c^2$, as required.
	\end{proof}

	\begin{lem} \label{lem:sum-skewed-squares}
		For every $c$, for every large enough $k$ which is divisible by $4$, there exist $a_1, \ldots, a_{80} \ge c$ such that $k = \sum_{i \in [80]} a_i(a_i + 1)$. 
	\end{lem}

	\begin{proof}
		By \Cref{lem:twenty-squares}, the integer $k/4$ can be written as the sum of twenty squares larger than $(c+1)^2$. Write $k/4 = \sum_{i =1}^{20} x_i^2$, where $x_i \ge c+1$. Then $k = \sum_{i=1}^{20} 4x_i^2$, i.e., $k$ is the sum of $20$ even squares larger than $4(c+1)^2$. 
		Observe now that every even square larger than $4(c+1)^2$ can be written as a sum $\sum_{i=1}^4 a_i (a_i + 1)$ with $a_i \ge c$, as $4a^2 = 2 \cdot \big( a(a+1) + a(a-1) \big)$.
		Thus, $k$ is a sum $\sum_{i=1}^{80} a_i(a_i + 1)$ with $a_i \ge c$, as required.
	\end{proof}

\section{Proof of approximate result} \label{sec:approximate}
    
    The aim in this section is to prove the following result, that allows us to find a cycle of almost the required number of chords, in any graph whose chromatic number is much larger than the clique number. 
    
    \begin{thm} \label{thm:main-almost}
        Let $k$ be large enough. Then there is a function $f$ such that for every graph $G$ either $\chi(G) \le f(\omega(G))$ or $G$ has a cycle with exactly $k'$ chords for some $k' \in \{k, k+1, k+2, k+3\}$.
    \end{thm}
    
    The proof relies on results from previous sections, as well as the following two lemmas. 
    The next lemma allows us to assume that there is a large collection of pairwise disjoint large induced balanced bipartite subgraphs with no edges between them.
    
    \begin{lem} \label{lem:find-disjoint-complete-bip}
    	Every graph $G$ contains either $\ell$ pairwise vertex-disjoint induced copies of $K_{\ell, \ell}$ with no edges between them or an induced subgraph $H$ with $\chi(H)\geq \chi(G)/(2\ell^2)-1$ such that either $\omega(H) < \omega(G)$ or it is induced $K_{\ell, \ell}$-free.
    \end{lem}
    
    \begin{proof}
        Let $K_1$, $\ldots$, $K_t$ be a maximal collection of pairwise vertex-disjoint induced copies of $K_{\ell, \ell}$ with no edges between them; write $K = V(K_1 \cup \ldots \cup K_t)$. We are done if $t \ge \ell$, so suppose that $t< \ell$. 
        Notice that one of the subgraphs $G[N(v)\cup \{v\}]$ for some $v \in  K$ or the graph $G \setminus (K \cup N(K))$ has chromatic number at least $\chi(G)/(2\ell^2)$, since the union of these $2\ell t+1$ graphs covers all the vertices in $G$ and $2\ell t + 1 \le 2\ell^2$.
    
        Suppose first that some $G[N(v)\cup\{v\}]$ has chromatic number at least $\chi(G)/(2\ell^2)$.
        Then $\chi(G[N(v)])$ is at least $\chi(G)/(2\ell^2)-1$ and $\omega(G[N(v)])<\omega(G)$, so we can take $H:=G[N(v)]$.
        Otherwise, if $G \setminus (K \cup N(K))$ has chromatic number at least $\chi(G)/(2\ell^2)$, then $H:=G \setminus (K \cup N(K))$ is induced $K_{\ell,\ell}$-free subgraph by maximality of $t$. 
    \end{proof}
    
    The next lemma is the key ingredient in proving \Cref{thm:main-almost}, whose proof will be given in \Cref{subsec:finding-cycle-almost-k}.
    
 	\begin{lem} \label{lem:finding-cycle-from-Kll}
    	Let $\ell\gg k\gg 1$,\footnote{concretely, it suffices to take $k \ge 10^{12}$. 
    	}
    	and let $p > 300$.
		In a $p$-extraction of a graph $G$, suppose that there are $101$ induced copies of $K_{\ell, \ell}$ with no edges between them. 
		Then there is a cycle with $k'$ chords, for some $k' \in \{k, k+1, k+2, k+3\}$. 
	\end{lem}
	
	We now prove \Cref{thm:main-almost} by using the previous lemmas.
    
    \begin{proof}[Proof of \Cref{thm:main-almost}]
        We prove by induction on $\omega$ that there exists $f(\omega)$ such that, for a graph $G$ with clique number $\omega$, if $\chi(G) > f(\omega)$ then $G$ has a cycle with exactly $k'$ chords for some $k' \in \{k, k+1, k+2, k+3\}$. 
        For $\omega = 1$, we can take $f(1) = 2$, for which the statement is vacuously true as there are no graphs with clique number $1$ and chromatic number at least $2$.
        
        Suppose that for $\omega_0 \ge 1$, $f(\omega)$ is defined for $\omega \le \omega_0$ so that the above statement holds.
        Let $p$ and $\ell$ be sufficiently large, e.g.,\ $p\geq 300$ and $\ell\geq 2^{500}\sqrt{k}$, and let $g$ be a function as in \Cref{cor:bt-induced-bip} for $k$ and $\ell$; namely, if $\chi(G) > g(\omega(G))$ then $G$ contains either an induced $K_{\ell,\ell}$ or a cycle with exactly $k$ chords.
        Now set $f(\omega_0+1) = 2^{p+1} \ell^2 \cdot \left(\max\{g(\omega_0), f(\omega_0)\} + 1\right)$. 
        Consider a graph $G$ with $\omega(G) = \omega_0+1$ and $\chi(G) > f(\omega_0+1)$. Our goal is to show that $G$ contains a cycle with exactly $k'$ chords, for some $k' \in \{k, k+1, k+2, k+3\}$. 
        Let $G'$ be a $p$-extracted graph of $G$. In particular, $\chi(G')\geq \chi(G)/2^p$.
        By \Cref{lem:find-disjoint-complete-bip}, one of the following three cases holds for $G'$.
        \begin{enumerate}[label = \rm (G\arabic*)]
            \item \label{itm:case-a}
                there are $\ell$ pairwise vertex-disjoint induced copies of $K_{\ell, \ell}$ with no edges between them,
            \item \label{itm:case-b}
                there is an induced subgraph $H \subseteq G'$ with $\chi(H) \ge \chi(G') /(2\ell^2) - 1$ and $\omega(H) < \omega(G')$,
            \item \label{itm:case-c}
                there is an induced subgraph $H \subseteq G'$ with $\chi(H) \ge \chi(G') /(2\ell^2) - 1$ which is $K_{\ell, \ell}$-free.
        \end{enumerate}
        If \ref{itm:case-b} holds, then $\chi(H) \ge \chi(G)/(2^{p+1}\ell^2) - 1 > f(\omega_0)$ and $\omega(H) \le \omega_0$. Thus, by the inductive hypothesis, $H$ contains a cycle with exactly $k'$ chords, for some $k' \in \{k, k+1, k+2, k+3\}$.
        
        If \ref{itm:case-c} holds, then $\chi(H) \ge \chi(G)/(2^{p+1}\ell^2) - 1 > g(\omega_0)$. Since $H$ is $K_{\ell, \ell}$-free in this case, by choice of $g$ there is a cycle with exactly $k$ chords in $H$.
        
        We may now assume that \ref{itm:case-a} holds. In particular, there are 101 copies of $K_{\ell, \ell}$ in $H$ that are pairwise vertex-disjoint with no edges between them. By \Cref{lem:finding-cycle-from-Kll}, using the choice of $G'$ as a $p$-extraction of $G$, there is a cycle in $G$ with exactly $k'$ chords, with $k\leq k'\leq k+3$.
    \end{proof}

    \subsection{Few edges between unimodal paths} \label{subsec:few-edges-unimodal}
    
    For the proof of \Cref{lem:finding-cycle-from-Kll}, we need the following two lemmas that allow us to assume that there are only few edges between two unimodal paths connected to a large induced balanced bipartite subgraph.

	\begin{lem}\label{lem:sqrtk-same-Kll}
	    Let $\ell \gg k$, $p \ge 1$, and let $K$ be an induced copy of $K_{\ell, \ell}$ in a $p$-extraction of a graph~$G$.
	    If $P$ is a unimodal path starting in $K$ and $u$ is a vertex outside of $P$ with an edge into $K \setminus V(P)$, then either there is a cycle in $G$ with exactly $k$ chords or there are at most $8\sqrt k$ edges from $u$ to $P$.
	\end{lem}
	\begin{proof}
        Let $x$ and $y$ be the ends of $P$ with $x\in V(K)$ and let $x^-$ and $y^-$ be the unique neighbours of $x$ and $y$ in $P$. By unimodality, vertices in $K$ do not send any edges to $P \setminus \{x, x^-, y, y^-\}$. Let $w$ be a neighbour of $u$ in $K$ distinct from $x$, which exists by the assumption on $u$.
        
        \begin{figure}[h]
            \centering
            \includegraphics{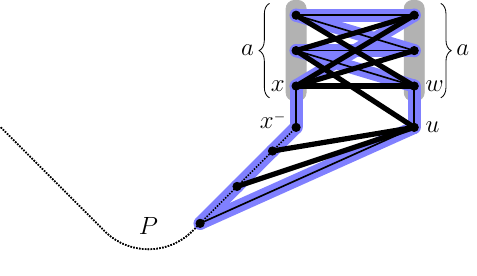}
            \vspace{-.3cm}
            \caption{\Cref{lem:sqrtk-same-Kll}}
            \label{fig:sqrtk-same-Kll}
        \end{figure}
        
        Suppose that $u$ sends more than $8\sqrt{k}$ edges to $P$. Let $Q$ be a path in $K$ with ends $x$ and $w$ on $2a$ or $2a+1$ vertices, where $a = \floor{\sqrt{k} - 2}$. 
        Let $k'$ be the number of chords in the path $Q$. 
        Then $k'$ is between $a^2 - (2a-1) = (a-1)^2$ and $a(a+1) - 2a = a^2 - a$. In particular,
        \begin{align*}
            k' 
            \ge (a-1)^2 \ge (\sqrt{k} - 4)^2 \ge k - 8\sqrt{k}. 
        \end{align*}
        Let $k''$ be the number of chords in the path $uwQxx^-$. Then, as $e_G(\{x^-, u\}, V(Q))\leq 2(2a+1)$, 
        \begin{align*}
            k'' = k'+e_G(\{x^-, u\}, V(Q)) - 2 \leq a^2 + 3a \le (a+2)^2 \le k.
        \end{align*}
        Take $b = k - k''$, so that $0 \le b \le 8\sqrt{k}$. Let $P'$ be the subpath of $P$ that starts at $x$ and ends at the $(b+1)$-th neighbour $u_b$ of $u$ in $P \setminus x$. Then $u w Q x P' u_b u$ is a cycle with exactly $k$ chords. 
    \end{proof}
    
    \begin{lem} \label{lem:sqrtk-different-Kll}
        Let $\ell \gg k \gg 1$, $p \ge 3$, and let
        $K$ and $K'$ be vertex-disjoint induced copies of $K_{\ell, \ell}$, with no edges between them, in a $p$-extraction of a graph $G$.
        For a unimodal path $P$ that starts in $K$ and is not in the last two layers, let $u$ be a vertex with a neighbour $w \in K' \setminus V(P)$. Then either there is a cycle in $G$ with exactly $k$ chords, or $u$ sends at most $30\sqrt{k}$ edges to $P$.
    \end{lem}
    
    \begin{proof}
        Let $x, y$ be the ends of $P$, where $x \in K$, and let $x^-,y^-$ be the neighbours of $x,y$ in $P$.
        Take $z$ to be any vertex in $K$ other than $x$ and take $v$ to be any vertex in $K'$ in the opposite side to $w$. 
        Let $Q$ be a unimodal path with ends $z$ and $v$ from a later layer than $P$ and a different layer than $u$, and let $z^-, v^-$ be the neighbours of $z, v$ in $Q$, respectively. 
        By choice of $Q$, there are no edges between $P \setminus \{x, x^-, y, y^-\}$ and $Q$ (see \Cref{obs:unimodal}).

        \begin{figure}[h]
            \centering
            \includegraphics{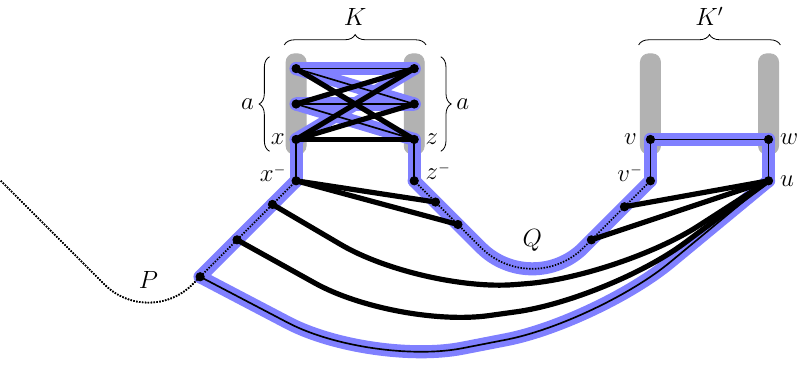}
            \caption{\Cref{lem:sqrtk-different-Kll}}
            \label{fig:sqrtk-diff-Kll}
        \end{figure}
        
        Suppose that there is no cycle with exactly $k$ chords and that $u$ has more than $30 \sqrt{k}$ neighbours in $P$.
        By \Cref{lem:sqrtk-same-Kll}, there are at most $8\sqrt{k}$ edges between $x^-$ and $Q$ and at most $8\sqrt{k}$ edges between $u$ and $Q$.
        Let $Q'$ be a path in $K$ with ends $x$ and $z$ on either $2a$ or $2a+1$ vertices, where $a = \floor{\sqrt{k} - 13}$.
        Denote by $k'$ the number of chords in the path $R := uw vQz Q' x x^-$, excluding the edge $u x^-$ if it exists. 
        Then 
        \begin{align*}
            k' \ge \,\, & \#\{\text{chords in $Q'$}\} \ge (a-1)^2 \ge (\sqrt{k} - 15)^2 \ge k - 30 \sqrt{k}.
        \end{align*}
        As $x^-, z^-, v^-$, and $u$ are the only vertices in $V(R)$ that may have neighbours in $Q'$,
        \begin{align*}
            k' \leq \,\, & \#\{\text{chords in $Q'$}\} + e_G(\{x^-, z^-, v^-, u\}, V(Q')) + 
             e_G(\{x^-, u\}, V(Q)) + e_G(\{x^-, z^-, v^-, v, u, w\}) \\
            \le \,\, & (a^2 - a) + 4 \cdot (2a + 1) + 16 \sqrt{k} + \binom{5}{2}  
            \le (a + 4)^2 + 16\sqrt{k} \le (\sqrt{k} - 9)^2 + 16\sqrt{k}\\
            = \,\, & k - 2 \sqrt{k} + 81 \le k, 
        \end{align*}
        where the last inequality follows from the assumption that $k$ is large.
        Take $b = k - k'$, so that $0 \le b \le 30 \sqrt{k}$.
        Let $P'$ be the subpath of $P$ that starts at $x$ and ends at the $(b+1)$-th neighbour~$u_b$ of $u$ in $P \setminus \{x, x^-\}$. Then, as there are no edges between $P \setminus \{x, x^-, y, y^-\}$ and $K$, $K'$ or $Q$,
        the cycle $uwvQzQ'xP'u_b u$ has exactly $k$ chords.
    \end{proof}

    \subsection{Proof of \Cref{lem:finding-cycle-from-Kll}}	 \label{subsec:finding-cycle-almost-k}
    
    We now prove \Cref{lem:finding-cycle-from-Kll}. Roughly speaking, the idea is to find many disjoint induced $K_{\ell, \ell}$'s with no edges between them, and join them via unimodal paths, to obtain a cycle whose length we can control by choosing subpaths of the $K_{\ell, \ell}$'s of appropriate lengths. Variants of the arguments used in this proof will appear in the following section.

	\begin{proof} [Proof of \Cref{lem:finding-cycle-from-Kll}]
	Let $K_0, \ldots, K_{100}$ be a collection of vertex-disjoint induced copies of $K_{\ell, \ell}$ in the $p$-extraction of $G$ such that there are no edges across distinct copies. Denote the bipartition of $K_i$ by $\{U_{i, 1}, U_{i, 2}\}$.
	Let $P_{i, j}$ be vertex-disjoint unimodal paths from $U_{i, j}$ to $U_{0, j}$, for $i \in [100]$ and $j \in [2]$.
	We take all the unimodal paths from different layers that are not the last two. 
	Let $u_{i,j}$ be the first internal vertex on $P_{i,j}$ from $K_i$, for each $i \in [100]$ and $j \in [2]$. 
	Denote by $u_{i,j}'$ the unique neighbour of $u_{i,j}$ in $V(P_{i,j})\cap V(K_i)$. That is, the end vertex of $P_{i,j}$ in $K_i$.
	A vertex $u$ is said to be \emph{complete} (resp.\ \emph{anti-complete}) to $U_{i,j}$ if it is adjacent to all (resp.\ none) of the vertices in $U_{i,j}\setminus\{u_{i,j}'\}$, for $i \in [100]$ and $j \in [2]$. 
	
	We first claim that, by possibly shrinking $\ell$ to $\ell'=\ell/2^{200}$, we may assume that each $u_{i,j}$ is either complete or anti-complete to $U_{s,t}$, for each $i, s \in [100]$ and $j, t \in [2]$.
	For each $v\in \bigcup_{s=1}^{100}(U_{s,1}\cup U_{s,2})$, write a $0$-$1$ vector $x_v\in\{0,1\}^{200}$ to encode its adjacency to $u_{i,j}$. That is, the $(i,j)$-coordinate is 1 if $vu_{i,j}\in E(G)$ and 0 otherwise.
	Then each $U_{s, t}\setminus\{u_{s, t}\}$ can be partitioned into at most $2^{200}$ subsets according to the value of $x_v$. Replacing $U_{s, t}$ by the largest amongst these subsets and adding $u_{s, t}$ suffices for our purpose.
	
	Let $v_{i,j}$ be the first internal vertex on $P_{i,j}$ from $K_0$, and let $v_{i,j}'$ be the unique vertex in $V(P_{i, j}) \cap V(K_0)$. 
	By using the same argument possibly shrinking~$\ell$ even further, we may assume that each $v_{i,j}$ is either complete or anti-complete to $U_{s, t}$, for $i, s \in [100]$ and $j, t \in [2]$.
	
	Our plan is to make a cycle in the following way. We will choose integers $a_1,\ldots,a_{100} \ge 0$, depending on $k$ and the structure we have just found. Starting from $v_{1,1}' \in U_{0,1}$, we take $P_{1,1}$ to reach $u_{1,1}'\in U_{1,1}$, then go through a path of length $2a_1+1$ in $K_1$ to reach $u_{1,2}'\in U_{1,2}$. The journey goes back to $K_0$ through $P_{1,2}$. Then we move to $v_{2,1}'$ to iterate. 
	After 100 iterations, we close the cycle by moving from $v_{100,2}'$ to $v_{1,1}'$ by an edge.
	
	\begin{figure}
	    \centering
	    \includegraphics{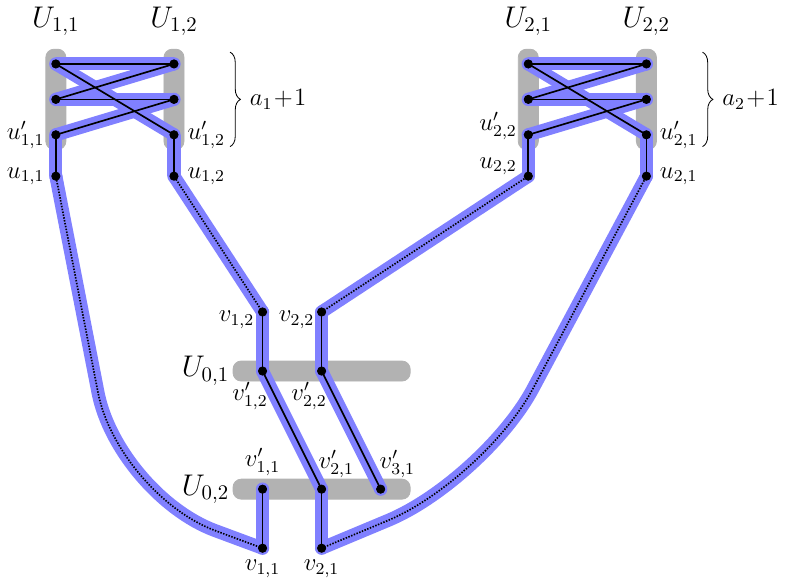}
	    \caption{Part of the cycle $\C(a_1, \ldots, a_{100})$}
	    \label{fig:cycle-from-Kll}
	\end{figure}
	
	Let us calculate how many chords exist in such a cycle, which we call $\C(a_1,\ldots,a_{100})$ (note that the number of chords depends only on $a_1, \ldots, a_{100}$, and not on the choice of the paths in $K_i$ for $i \in [100]$, by previous assumptions). Firstly, a $(2a_i+1)$-edge path in a copy of $K_{a_i,a_i}$ gives $(a_i + 1)^2 - (2a_i+1) = a_i^2$ chords and the path in $K_0$ has at most $100^2$ chords. Second, each $u_{i,j}$ or $v_{i,j}$ that is complete to $U_{s, t}$ contributes $a_s$ chords (to $U_{s, t} \setminus \{u_{s, t}'\}$), and there are at most $100^2$ chords between vertices $u_{i, j}$ or $v_{i, j}$ and the path in $K_0$. Third, each $u_{i,j}$ or $v_{i,j}$ adjacent to $u'_{s,t}$, with $(i,j)\neq(s,t)$, contributes one chord. Finally, there are at most $O(\sqrt{k})$ chords between internal vertices of $P_{i,j}$s (by \Cref{lem:sqrtk-different-Kll,lem:sqrtk-same-Kll}, because otherwise there is a cycle with exactly $k$ chords and then, we are done). 
	Overall, the cycle $\C(a_1, \ldots, a_{100})$ has the following number of chords
	\begin{align}\label{eq:chordcount}
	    \sum_{s=1}^{100} a_s^2 +\sum_{s=1}^{100} t_s a_{s} + O(\sqrt{k}),
	\end{align}
	where $t_s$ is the number of vertices $u_{i, j}$ or $v_{i, j}$ that are complete to $U_{s, 1}$ plus the number of vertices $u_{i, j}$ or $v_{i, j}$ that are complete to $U_{s, 2}$, and $O(\sqrt{k})$ is a function that only depends on the graph structure that we have found, i.e., the 101 induced $K_{\ell,\ell}$'s and the unimodal paths in between.
	
	In~\eqref{eq:chordcount}, the $O(\sqrt{k})$ term and the terms $t_s$, with $s \in [100]$, are all fixed parameters, i.e., they do not depend on the choice of $a_1, \ldots, a_{100}$. Also $t_s \le 400$ for each $s \in [100]$.
	Thus, one may write~\eqref{eq:chordcount} as
	\begin{align} \label{eq:chordcount-f}
	    f_G(k)+\sum_{s=1}^{100}(a_s^2 +t_sa_s),
	\end{align}
	where $f_G(k)=O(\sqrt{k})$ is a parameter depending on $k$ and $G$, but not on $a_1, \ldots, a_{100}$. 
	
	Among $t_1, \ldots, t_{100}$, at least $20$ values are even, or at least $80$ are odd. If the former happens, we may assume $t_1,\ldots,t_{20}$ are the even numbers by relabelling the indices. Then, by \Cref{lem:twenty-squares}, we can choose $b_s\geq 200$ for $s \in [20]$, such that
	\begin{align*}
	    \sum_{s = 1}^{20} b_s^2 = k - f_G(k) + \sum_{s=1}^{20}\frac{t_s^2}{4}.
	\end{align*}
	Indeed, large enough $k$ guarantees the right-hand side is a large enough positive integer that can be expressed as the sum of exactly $20$ squares larger than $200$.
	Now let $a_s = b_s-t_s/2$ for $s \in [20]$ and 0 otherwise (notice that $a_2 \ge 200 - t_s/2 \ge 0$). Then the cycle $\C(a_1,a_2,\ldots,a_{100})$ has exactly $k$ chords, since~\eqref{eq:chordcount-f} becomes
	\begin{align*}
	    f_G(k)+\sum_{s=1}^{100} (a_s^2+t_sa_s)
	    =f_G(k) + \sum_{s = 1}^{20} \left( b_s^2 - \frac{t_s^2}{4}\right) = k.
	\end{align*}
	Otherwise, there are at least $80$ odd $t_s$'s, say $t_1,\ldots,t_{80}$. Let $r$ be the unique integer divisible by $4$ such that 
	\begin{align*}
	    r-3 \le k-f_G(k)+\sum_{s=1}^{80}\frac{t_s^2-1}{4} \le r.
	\end{align*}
	By our choice of $t_1,\dots,t_{80}$, $(t_s^2 - 1)/4$ is an integer for each $s\in[80]$.
	\Cref{lem:sum-skewed-squares} then shows that, since $r$ is a large enough integer divisible by $4$, there exist integers $b_s \ge 300$, for $s \in [80]$, such that
	\begin{align*}
	    r = \sum_{s=1}^{80} b_s(b_s+1).
	\end{align*}
	Now take $a_s = b_s - (t_s - 1)/2$ if $s \in [80]$ (so $a_s \ge 300 - (t_s-1)/2 \ge 0$), and $0$ otherwise. Then the number $k'$ of chords in $\C(a_1,a_2,\ldots,a_{100})$, using \eqref{eq:chordcount-f}, is 
	\begin{align*}
	    k' 
	    & = f_G(k) + \sum_{s=1}^{100}(a_s^2 + t_s a_s) \\
	    & = f_G(k) + \sum_{s=1}^{80}\left(b_s(b_s+1) - \frac{t_s^2-1}{4}\right)
	    = f_G(k) + r - \sum_{s=1}^{80}\frac{t_s^2-1}{4}.
	\end{align*}
	By our choice of $r$, $k \le k' \le k+3$. Therefore, there is a cycle with exactly $k'$ chords for some $k' \in \{k, k+1, k+2, k+3\}$, as required.
	\end{proof}
	
\section{Exact result} \label{sec:exact}
    In this section we will prove \Cref{thm:main}, our main result. To do so, we will prove \Cref{lem:find-cycle-from-Kll-exact} below, which asserts that if a $p$-extraction of a graph $G$ contains $\ell$ pairwise disjoint induced copies of $K_{\ell, \ell}$ with no edges between them, where $p, \ell \gg k \gg 1$, then there is a cycle with exactly $k$ chords. Notice that \Cref{thm:main} can be deduced from this lemma, along with \Cref{cor:bt-induced-bip,lem:find-disjoint-complete-bip}, following the proof of \Cref{thm:main-almost} in \Cref{sec:approximate} and replacing the call to \Cref{lem:finding-cycle-from-Kll} by a call to \Cref{lem:find-cycle-from-Kll-exact}.
    
    \subsection{Overview of the proof}
        The basic idea for the proof of \Cref{lem:find-cycle-from-Kll-exact} is similar to that of \Cref{lem:finding-cycle-from-Kll}. Given $\ell$ induced copies $K_1, \ldots, K_{\ell}$ of $K_{\ell, \ell}$ with no edges between them, we clean them up so that we can build cycles, taking paths $P_i \subseteq K_i$ and connecting them by unimodal paths, in such a way that the number of chords depends only on the lengths of the $P_i$'s. However, as seen in the previous section, in some cases a parity issue may cause this strategy to fail to give the precise desired number of chords. 
        
        To fix this, when cleaning up the $K_i$'s we take into account, for each vertex $u$ in one of the $K_i$'s and each triple $T = \{j_1, j_2, j_3\}$ of layers (with $j_1 > j_2 > j_3$), a $j_1$-father $f_1(u;T)$ of $u$, a $j_2$-father $f_2(u;T)$ of $f_1(u;T)$, and a $j_3$-father $f_3(u;T)$ of $f_2(u;T)$ (recall that in the previous section we considered just one father of one vertex per $K_{\ell, \ell}$). 
        
        \Cref{lem:cleaning-step-one} allows us to assume that there are no edges between $\{f_1(u; T), f_2(u; T), f_3(u;T)\}$ and $K_i$'s not containing $u$. An important step towards the proof of \Cref{lem:cleaning-step-one} is \Cref{lem:manyverticeswithmanyneighboursinKmms}. The latter lemma considers a case where there is a collection of 31 pairwise vertex-disjoint induced $K_{\ell,\ell}$'s with no edges between them, and a collection of 31 vertices that are each complete to each of these $K_{\ell, \ell}$'s, or are each complete to one part of each $K_{\ell, \ell}$ and anti-complete to the other. The proof of this lemma follows the usual scheme of building a cycle and controlling the number of its chords by controlling the number of vertices used from each $K_{\ell, \ell}$. 
        The proof here is, in fact, a bit simpler as instead of connecting the $K_{\ell, \ell}$'s using unimodal paths we can use the 31 vertices.
        
        With the above assumption at hand, namely that there are no edges between the set of parents $\{f_1(u;T), f_2(u;T), f_3(u;T)\}$ and $K_i$'s not containing $u$, we now wish to analyse the interaction between any $K_i$ and the sets $\{f_j(u;T) : u \in K_i\}$, $j = 1, 2, 3$, for any fixed $T$. This is done in \Cref{lem:cases}. The full statement of the lemma is quite long. In short, it allows us to assume that one of six cases holds, each describing a concrete well-structured graph. 
        
        We then turn to the proof of \Cref{lem:find-cycle-from-Kll-exact}, which is the main result of the section. By applying \Cref{lem:cases} to sufficiently many triples of layers, we can assume that one of the six cases mentioned above holds for enough $K_i$'s and enough layers. In four of these cases (which we analyse simultaneously), we build the cycles as usual and can reach the desired number of chords without encountering parity issues. In the remaining two cases, we require a few small gadgets to adjust the number of chords in case of a parity issue.
        
    \subsection{Removing cross edges}
    
        As mentioned above, our first step towards \Cref{lem:find-cycle-from-Kll-exact} is \Cref{lem:cleaning-step-one} which allows us to assume that, given a collection of induced $K_{\ell, \ell}$'s with no edges between them, there are no edges between parents of one copy of $K_{\ell, \ell}$ and any other copy of $K_{\ell, \ell}$. Before proving \Cref{lem:cleaning-step-one}, we prove the following lemma which considers the other extreme, where there are many `well-connected' vertices that are each fully joined to at least one part of each such $K_{\ell, \ell}$. 
        
        The proof is similar to other proofs in this paper, where we build a cycle using paths from each $K_{\ell,\ell}$. However, instead of using unimodal paths to connect these paths, we use the well-connected vertices, which simplifies the proof.
        
        \begin{lem}\label{lem:manyverticeswithmanyneighboursinKmms}
            Let $\ell \gg k \gg 1$. In a graph $G$, suppose that $K_1, \ldots, K_{31}$ are pairwise vertex-disjoint induced copies of $K_{\ell, \ell}$ with no edges between them. Denote by $\{U_{i,1}, U_{i,2}\}$ the bipartition of $K_i$.
            Let~$X$ be  set of  $31$ vertices such that one of the following holds:
            \begin{enumerate}[label = \rm(\alph*)]
                \item \label{itm:complete-both-sides}
                    every $x \in X$ is complete to $K_i$, for $i \in [31]$ or
                \item \label{itm:complete-one-side}
                    every $x \in X$ is complete to $U_{i,1}$ and anti-compete to $U_{i,2}$, for $i \in [31]$.
            \end{enumerate}
            Then there is a cycle in $G$ with exactly $k$ chords.
        \end{lem}
        
        \begin{proof}
            Suppose first that \ref{itm:complete-both-sides} is the case. Let $x_1, \ldots, x_{20} \in X$ be distinct. Let $x$ be the number of edges in the induced subgraph $G[\{x_1, \dots, x_{20}\}]$.
            For non-negative integers $a_1, \ldots, a_{20}$, let $Q_i$ be a path in $K_i$ of length $2a_i + 1$.
            Then each $x_i$ is adjacent to the ends of $Q_j$ for $i, j \in [20]$. Let $\C$ to be the cycle $x_1 Q_1 x_2 \ldots x_{20} Q_{20} x_1$. 
            \begin{figure}[h]
                \centering
                \includegraphics{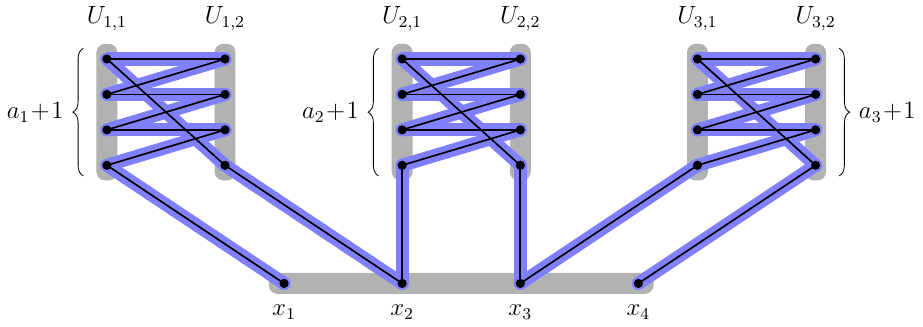}
                \caption{Part of the cycle $\C$ in Case \ref{itm:complete-both-sides}}
                \label{fig:complete-both-sides}
            \end{figure}
            
            We now wish to count the number of chords in $\C$, which only depends on $a_1,\dots,a_{20}$ and $x$. First, by counting the edges induced on $V(Q_i)$, the edges from $x_i$ to $Q_j$, and the edges induced on $\{x_1, \ldots, x_{20}\}$, the number of edges with both ends in $V(\C)$ is 
            \begin{equation*}
                \sum_{i = 1}^{20} (a_i+1)^2 + \sum_{i = 1}^{20} 20 \cdot 2(a_i+1) + x,
            \end{equation*}
            whereas the length of the cycle is $\sum_{i = 1}^{20}(2a_i + 3)$.       Subtracting the cycle length from the number of edges induced on $V(\C)$, the number of chords equals to
            \begin{equation*}
                \sum_{i = 1}^{20} (a_i^2 + 40a_i + 38) +x = 
                \sum_{i = 1}^{20} (a_i + 20)^2 +x- c,
            \end{equation*}
            where $c=20 \cdot (400 - 38)$.
            By using \Cref{lem:twenty-squares} for large enough $k$, there exist integers $b_1, \ldots, b_{20} \ge 20$ such that $\sum_{i = 1}^{20} b_i^2 = k -x+ c$. Given such $b_1, \ldots, b_{20}$, take $a_i = b_i - 20$ for $i \in [20]$. Then the cycle $\C$ has exactly $k$ chords, as required.
            
            Suppose now that \ref{itm:complete-one-side} happens. Let $x_1, \ldots, x_{21} \in X$ be distinct.  Let $x$ be the number of edges in the induced subgraph $G[\{x_1, \dots, x_{20}\}]$. Given integers $a_1, \ldots, a_{20} \ge 0$, let $Q_i$ be a path in $K_i$ of length $2a_i$, with both ends in $U_{i,1}$. Let $\C$ be the cycle $x_1 Q_1 x_2 \ldots x_{21} Q_{21} x_1$. 
            \begin{figure}
                \centering
                \includegraphics{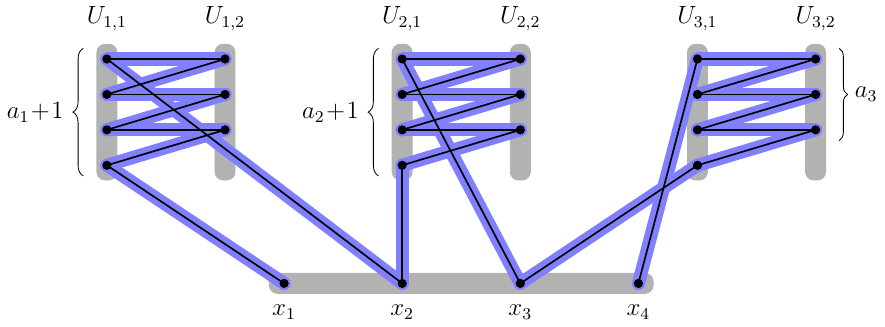}
                \caption{Part of the cycle $\C$ in Case \ref{itm:complete-one-side}}
                \label{fig:complete-one-side}
            \end{figure}
            
            Then the number of edges with both ends in the cycle $\C$ is
            \begin{equation*}
                \sum_{i = 1}^{21} a_i(a_i + 1) + \sum_{i = 1}^{21} 21 (a_i + 1)+x, 
            \end{equation*}
            where the length of the cycle is $\sum_{i = 1}^{21} (2a_i + 2)$.
            Thus, the number of chords in $\C$, obtained by subtracting the number of the $\C$-edges from the number of induced edges on $V(\C)$, is
            \begin{equation*}
                \sum_{i = 1}^{21} (a_i^2 + 20 a_i + 19) +x
                = \sum_{i = 1}^{21} (a_i + 10)^2 +x - c,
            \end{equation*}
            where $c=21 \cdot (100 - 19)$.
            Again by \Cref{lem:twenty-squares}, there exist $b_1, \ldots, b_{21} \ge 10$ such that $\sum_{i = 1}^{20} b_i^2 = k-x+ c$. Given such $b_1, \ldots, b_{21}$, choosing $a_i = b_i - 10$ yields a cycle $\C$ with exactly $k$ chords.
        \end{proof}

        The next lemma allows us to assume that there is a collection of many disjoint induced $K_{\ell,\ell}$'s with no edges between them, such that for many triples of layers $T = \{j_1 > j_2 > j_3\}$ and every vertex $u$ in one of these $K_{\ell,\ell}$'s, there are no edges between $\{f_1(u;T), f_2(u;T), f_3(u;T)\}$ and $K_{\ell,\ell}$'s not containing $u$ (where $f_1(u;T)$ is a $j_1$-father of $u$, $f_2(u;T)$ is a $j_2$-father of $f_1(u;T)$, $f_3(u;T)$ is a $j_3$-father of $f_2(u;T)$, and $T = \{j_1 > j_2 > j_3\}$). 
        
        Roughly speaking, to prove this lemma we choose a collection $\{f_1(u;T), f_2(u;T), f_3(u;T)\}$ for each $u$ and $T$. We then clean up the $K_{\ell,\ell}$'s using Ramsey's theorem to assume that each $f_j(u;T)$ is either complete or anti-complete to each part of each $K_{\ell,\ell}$. Finally, using the previous lemma we may assume that for the vast majority of choices of $u$, $T$ and $K_{\ell,\ell}$-copy $K$, there are no edges between $\{f_1(u;T), f_2(u;T), f_3(u;T)\}$ and $K$. The K\"ovari--S\'os--Tur\'an theorem then allows us to find the required structure.
        
        \begin{lem} \label{lem:cleaning-step-one}
            Let $p, \ell \gg m \gg k \gg 1$ and let $G_p$ be a $p$-extraction of $G$. 
            Suppose that, in $G_p$, $K_1, \ldots, K_{\ell}$ are pairwise vertex-disjoint induced copies of $K_{\ell, \ell}$ with no edges between them. Then either there is a cycle with exactly $k$ chords, or there exist a collection $K_1', \ldots, K_m'$ of pairwise vertex-disjoint induced $K_{m, m}$'s in $G_p$ and a subset $J \subseteq [p]$ of $m$ indices such that, for every triple $T=\{j_1, j_2, j_3\} \subseteq J$ with $j_1 > j_2 > j_3$, every $i \in [m]$ and every $u \in V(K_i')$, there exist vertices $f_{1}(u;T), f_{2}(u;T), f_{3}(u;T)$ as follows:
            \begin{itemize}
                \item $f_{1}(u;T)$ is a $j_1$-father of $u$, 
                \item $f_{2}(u;T)$ is a $j_2$-father of $f_{1}(u;T)$,
                \item $f_{3}(u;T)$ is a $j_3$-father of $f_{2}(u;T)$,
                \item there are no edges between $\{f_{1}(u;T),f_{2}(u;T),f_{3}(u;T)\}$ and $\bigcup_{s \neq i}V(K_{s}')$. 
            \end{itemize}
        \end{lem}
        
    \begin{proof}
        We shall choose $a,b, \ell'$ with $\ell,p\gg M \gg a\gg b\gg m\gg \ell'\gg k$.
        For each $u\in K_1\cup \dots\cup K_{\ell}$ and a triple $T=\{j_1,j_2,j_3\}$ in $[a]$ (with $j_1 > j_2 > j_3$), choose $f_i(u;T)$, for $i\in [3]$, arbitrarily from those vertices such that $f_1(u;T)$ is a $j_1$-father of $u$ and $f_i(u;T)$ is a $j_i$-father of $f_{i-1}(u;T)$ for $i \in \{2, 3\}$.
        Let $F(u)$ be the collection of all $f_i(u;T)$ chosen.
                    
        Let $\{X_i, Y_i\}$ be the bipartition of $K_{i}$. We claim that for $i \in [b]$ we can find subsets $X'_i\subseteq X_i$ and $Y'_i\subseteq Y_i$ of size $m$ so that for all $u\in X'_i\cup Y'_i$ and $j\neq i$, each $v\in F(u)$ is either complete or anti-complete to $X'_j$ and either complete or anti-complete to $Y'_j$. 
    
        This can be done by an analogous technique to the one used in the proof of \Cref{lem:finding-cycle-from-Kll}.
        Namely, for a vertex $w$ and a set $Z$, let $\mathbf{x}(w,Z)\in\{0,1\}^{Z}$ be the $0$-$1$ vector that encodes the adjacency between vertices in $Z$ and $w$. That is, the corresponding coordinate to $v\in Z$ is $1$ if $w$ and $v$ are adjacent and 0 otherwise.
        Then one can partition $X_i$ (resp.\ $Y_i$) according to the values of $\mathbf{x}(w,Z)$ with $w\in X_i$ (resp.\ $X_i$). By choosing the largest subset in the partition, we have subsets $X_i'\subseteq X_i$ and $Y_i'\subseteq Y_i$ such that $|X_i'|\geq 2^{-|Z|}|X_i|$ and $|Y_i'|\geq 2^{-|Z|}|Y_i|$, and each $z\in Z$ is either complete or anti-complete to each $w\in X_i'\cup Y_i'$.
    
        First, choose arbitrary subsets $X_b^1\subseteq X_b$ and $Y_b^1\subseteq Y_b$ with $|X_b^1| = |Y_b^1| = M$. Then for each $i<b$, it is possible to choose $X_i^1 \subseteq X_i$ and $Y_i^1\subseteq Y_i$ of size $M$ such that each vertex $v\in F(u)$, with $u\in X_j^1\cup Y_j^1$ and $j>i$, is either complete or anti-complete to $X_i^1\cup Y_i^1$ (indeed, the previous paragraph tells us that there exist such subsets of size at least $\ell \cdot 2^{-a^3 b M} \ge M$).
        
        Next, we repeat in the reverse order. That is, starting with $i = 1$, we choose subsets $X_i' \subseteq X_i^1$ and $Y_i'\subseteq Y_i^1$ of size $m$ such that each vertex $v\in F(u)$, with $u\in X_j'\cup Y_j'$ and $j<i$, is either complete or anti-complete to $X_i'\cup Y_i'$ (this is possible since we are guaranteed such subsets of size at least $M \cdot 2^{-a^3 b m} \ge m$). The sets $X_i'$ and $Y_i'$ satisfy the requirements.
    
        We then collect all `fathers' $v\in F(u)$ for some $u\in X_i'\cup Y_i'$.
        That is, we set $F:=\bigcup_{i=1}^{b}\bigcup_{u\in X'_i\cup Y'_i}F(u)$.
        Consider the auxiliary bipartite graph $B$ between $F$ and $[b]$ where $v\in F$ and $j\in [b]$ are adjacent whenever $v$ is complete to  $X'_j$ or $Y'_j$. Colour each edge $(v,j)$ by red, blue, or green if $v$ is complete to $X'_j$ only, $Y'_j$ only, or both $X'_j\cup Y'_j$, respectively.
        By Lemma~\ref{lem:manyverticeswithmanyneighboursinKmms}, if there is no cycle with exactly $k$ chords, then $B$ contains no monochromatic  $K_{\ell', \ell'}$. Since $m\gg \ell'$, the bipartite Ramsey Theorem tells us that $B$ contains no copy of $K_{m,m}$. 
        In particular, there are at most $m\binom bm$ vertices $f\in F$ with $\deg_B(f)\geq m$. Since $a\gg m\binom bm+m$, we can choose a set $J$ of $m$  indices so that for $j\in J$, all $j$-fathers $f\in F$ have $\deg_B(f)< m$. 
         For each $u\in \bigcup_{i=1}^bX_i'\cup Y_i'$, let $F'(u)\subseteq F(u)$ be the subset of fathers $f_i(u;T)$ with $i \in \{1,2,3\}$ and $T\in\binom{J}{3}$.
         Note that $|F'(u)|\leq |J|^3\leq m^3$.
        
        Consider an auxiliary directed graph $D$ on $[b]$ where we join $i$ to $j$ if there is some $u\in X'_i\cup Y'_i$ and $v\in F'(u)$ with $v$ complete to $X'_j$ or $Y'_j$. 
        As $\big|\bigcup_{u\in X'_{i}\cup Y'_i}F'(u)\big|\leq |X_i'\cup Y_i'|m^3 = 2m^4$ for each $i$ and each $f\in F'(u)$ is adjacent to at most $m$ of the sets $X_j\cup Y_j$, the digraph $D$ has maximum out-degree at most $2m^5\leq b/(2m+1)$. 
        Therefore, $D$ has an independent set $I$ of size $m$. Letting $K'_1, \dots, K'_m$ be the complete bipartite graphs induced by $X'_i\cup Y'_i$ for $i\in I$, we get a set of complete bipartite graphs satisfying the lemma. 
        \end{proof}
        
    \subsection{Analysing the structure of a complete bipartite subgraph and its parents}
    Our next task is to prove \Cref{lem:cases}, which will be used to analyse the interaction between a complete bipartite subgraph and its parents, `grandparents' and `great-grandparents'.
        
    The following simple lemma will be useful in what follows.
    \begin{lem} \label{obs:induced-matching}
        Let $G_2$ be a bipartite graph on the bipartition $A\cup B$, where $|A| = |B| = m$, and let $G_1$ be its subgraph. Suppose that $d_{G_1}(a) \ge 1$ for every $a \in A$ and $d_{G_2}(b) \le d$ for every $b \in B$.
        Then there exist $A' \subseteq A$ and $B' \subseteq B$, such that $|A'| = |B'| \ge m/4d^2$ and each induced subgraph $G_i[A', B']$, $i=1,2$, is a perfect matching. 
    \end{lem}
    \begin{proof}
        Let $A_0 := \{a \in A: d_{G_2}(a) \le 2d\}$. Then $e(G_2) \le d|B|$ implies $|A \setminus A_0| \le e(G_2)/2d \le |B|/2 = m/2$ and hence, $|A_0| \ge m/2$.
        Let $A' \subseteq A_0$ be a maximal subset such that there exists $B' \subseteq B$ such that each $G_i[A', B']$, $i=1,2$, induces a perfect matching. We claim that every vertex in $A_0 \setminus A'$ has a $G_2$-neighbour in $N_{G_2}(A')$. Indeed, if $a \in A_0 \setminus A'$ has no $G_2$-neighbour in $N_{G_2}(A')$ then it has a $G_1$-neighbour $b \in B \setminus N_{G_2}(A')$, which makes each $G_i[A' \cup \{a\}, B' \cup \{b\}]$, $i=1,2$, a perfect matching. This contradicts the maximality of $A'$.
        Thus, $A_0 \subseteq N_{G_2}(N_{G_2}(A'))$, which implies $|A_0| \le d |N_{G_2}(A')| \le 2d^2|A'|$. As $|A_0| \ge m/2$, $|B'| = |A'| \ge m/4d^2$, as desired.
    \end{proof}
    
    \begin{lem}\label{lem:cases}
        For $\ell \gg \omega$, let $K$ be an induced copy of $K_{\ell,\ell}$ with bipartition $X\cup Y$.
        Suppose that $R, S, T$ are pairwise disjoint sets outside $K$ such that every vertex in $K$ has a neighbour in~$R$, every vertex in $R$ has a neighbour in $S$, and every vertex in $S$ has a neighbour in $T$. 
        Let $U:=R\cup S\cup T$ and $\ell' := \log \ell / 10000$.

        Then there exist sets $X'$ and $Y'$ of size $\ell'$, each of which contained in a different set amongst $X$ and $Y$, sets $A, A', A'' \subseteq U$, each of which contained in a different set amongst $R, S, T$, and vertices $a, b \in U$ such that one of the following holds:
        \begin{enumerate} [label = (K\arabic*)]
            \item \label{itm:a}
                The vertex $a$ is complete to $X' \cup Y'$.
            \item \label{itm:b}
                The vertex $a$ is complete to $X'$ and anti-complete to $Y'$, and $b$ is complete to $Y'$ and anti-complete to $X'$. 
            \item \label{itm:c}
                Both the induced bipartite graphs $G[A, X']$ and $G[A', Y']$ are perfect matchings, $A$ is anti-complete to $Y'$, and $A'$ is anti-complete to $X'$.
                
            \item \label{itm:d}
                The induced bipartite graph $G[A, X']$ is a perfect matching and $A$ is complete to $Y'$.
            \item \label{itm:e}
                The induced bipartite graph $G[A, X']$ is a perfect matching, $A$ is anti-complete to $Y'$, $A'$ is anti-complete to $X'$ and either complete or anti-complete to $Y'$, every vertex in $A$ has a neighbour in $A'$, and $b$ is complete to $Y'$ and anti-complete to $X'$.
                
            \item \label{itm:f}
                Each component in $G[A, A', A'', X']$ consists of four vertices $u, u', u'', x$ in $A, A', A'', X'$, respectively, such that all pairs in $\{u, u', u'', x\}$, except for possibly $uu''$, form edges. Additionally, $A \cup A' \cup A''$ is anti-complete to $Y'$, and $b$ is complete to $Y'$ and anti-complete to $X'$.
                
        \end{enumerate}
    \end{lem}
    
    \begin{proof}
        We begin by proving the following claim.
        \begin{claim}
            Let $\ell_0 := \ell^{1/7} / 10$.
            There exist $R_1, S_1, T_1 \subseteq U$ of size $\ell_0$, each of which is contained in a different set amongst $R, S, T, Q_1 \subseteq X$ of size $\ell_0$, and a vertex $u_1 \in U$ such that one of the following holds:
            \begin{enumerate}[label = \rm(\alph*)]
                \item \label{itm:case1}
                     The vertex $u_1$ is complete to $Q_1$.
                \item \label{itm:case2}
                     The vertex $u_1$ is complete to $R_1$ and anti-complete to $Q_1$ and $G[Q_1, R_1]$ induces a perfect matching.
                \item \label{itm:case3}
                    Every component in $G[Q_1, R_1, S_1]$ is an induced path $qrs$ with $q, r, s$ in $Q_1, R_1, S_1$, respectively.
                \item \label{itm:case4}
                    Every component in $G[Q_1, R_1, S_1, T_1]$ consists of four vertices $q, r, s, t$ from $Q_1, R_1, S_1, T_1$, respectively, such that all pairs in $\{q, r, s, t\}$ are edges, except for possibly $rt$.
            \end{enumerate}
        \end{claim}
        
        \def \nq{q}
        \def \nr{r}
        \def \ns{s}
        \def \nt{t}
             
        \begin{proof}[Proof of the claim]
            We shall use $\ell = (10\ell_0)^7$  implicitly throughout the proof.
            
            If a vertex in $U$ has at least $\ell_0$ neighbours in $X$ then we may take $u_1$ to be this vertex and $Q_1$ to be a set of $\ell_0$ of its neighbours in $X$, which satisfies \ref{itm:case1}.
            In this case, the sets $R_1, S_1, T_1$ are chosen arbitrarily, as they play no roles in what follows. We may hence assume that every vertex in $U$ sends at most $\ell_0$ edges to $X$.
            
            By \Cref{obs:induced-matching}, there are subsets $Q^{(1)} \subseteq Q$ and $R^1 \subseteq R$ with $|Q^{(1)}| = |R^{(1)}| \ge \ell / 4\ell_0^2$ such that $G[Q^{(1)}, R^{(1)}]$ induces a perfect matching. 
            For $r \in R^{(1)}$, denote by $\nq(r)$ the unique neighbour of $r$ in $Q^{(1)}$; define $\nr(q)$ similarly for each $q \in Q^{(1)}$.
            
            Suppose that there is a vertex $u \in S \cup T$ with at least $2\ell_0$ neighbours in $R^{(1)}$. 
            As $u$ has at most~$\ell_0$ neighbours in $Q$, there are at least $\ell_0$ neighbours $r$ of $u$ such that $\nq(r)$ is not adjacent to $u$.
            Let $R_1 \subseteq R^{(1)}$ be a set of $\ell_0$ such vertices $r$.
            Then taking $Q_1 := \{\nq(r) : r \in R_1\}$ and $u_1 := u$ proves \ref{itm:case2}.
            Thus, we may assume that every vertex in $S \cup T$ has at most $2\ell_0$ neighbours in $R^{(1)}$.
            
            Let $S^{(1)} \subseteq S$ be a set of size $|R^{(1)}|$ such that each vertex in $R^{(1)}$ has a neighbour in $S^{(1)}$. 
            Let $G_1 = G[R^{(1)}, S^{(1)}]$ and let $G_2$ be the bipartite graph on $R^{(1)}\cup S^{(1)}$ where $rs$, $r \in R^{(1)}$ and $s \in S^{(1)}$, is an edge if $s$ has a neighbour in $\{r, \nq(r)\}$. By \Cref{obs:induced-matching}, 
            there exist subsets $R^{(2)} \subseteq R^{(1)}$ and $S^{(2)} \subseteq S^{(1)}$ with $|S^{(2)}| = |R^{(2)}| \ge |R^{(1)}| / 36\ell_0^2 \ge \ell / 144 \ell_0^4$, such that $G_i[R^{(2)}, S^{(2)}]$ induces a perfect matching, for $i=1,2$. 
            
            Let $Q^{(2)} := \{\nq(r) : r \in R^{(2)}\}$  and let $\nr(s), \ns(r), \nq(s), \ns(q)$ be defined similarly as above for $s \in S^{(2)}$, $ r \in R^{(2)}$, and $q \in Q^{(2)}$. That is, $\nr(s)$ is the unique neighbour of $s$ in $R^{(2)}$; $\nq(s)$ is the unique neighbour of $\nr(s)$ in $Q^{(2)}$; etc. 
            Then each component in $G[Q^{(2)}, R^{(2)}, S^{(2)}]$ consists of vertices $q, r, s$ with $q \in Q^{(2)}, r \in R^{(2)}, s \in S^{(2)}$ such that $qr$ and $rs$ are edges. 
            If $qs$ is a non-edge for at least half of the components, then \ref{itm:case3} holds. We may thus assume that $qs$ is an edge for at least half of the components, implying that there are subsets $Q^{(3)} \subseteq Q^{(2)}$, $R^{(3)} \subseteq R^{(2)}$ and $S^{(3)} \subseteq S^{(2)}$, with $|Q^{(3)}| = |R^{(3)}| = |S^{(3)}| \ge \ell / 288 \ell_0^4$, such that $G[Q^{(3)}, R^{(3)}, S^{(3)}]$ induces a $K_3$-factor.
            
            Suppose that $u \in T$ has at least $2\ell_0$ neighbours in $S^{(3)}$. 
            Let $R_1\subseteq S^{(3)}$ be a set of $\ell_0$ neighbours such that $\nq(s)$ is not a neighbour of $u$. Indeed, such a set exists because $u$ has at most $\ell_0$ neighbours in $Q^{(3)}$. Then taking $Q_1 := \{\nq(s) : s \in R_1\}$ and $u_1 := u$ makes \ref{itm:case2} hold.
            We thus assume that every vertex in $T$ has at most $2\ell_0$ neighbours in $S^{(3)}$.
            
            Now \Cref{obs:induced-matching} implies that there exist subsets $Q^{(4)}, R^{(4)}, S^{(4)}, T^{(4)}$ of $Q^{(3)}, R^{(3)}, S^{(3)}, T$, respectively, which are of the same size at least $|Q^{(3)}| / 100\ell_0^2 \ge \ell / 28800\ell_0^6$, such that each component in $G[Q^{(4)}, R^{(4)}, S^{(4)}, T^{(4)}]$ consists of vertices $q, r, s, t$ from $Q, R, S, T$, respectively, such that $qr, qs, rs, ts$ are edges. If $qt$ is a non-edge for at least half of the components, then \ref{itm:case3} holds by taking subsets of $Q^{(4)}, S^{(4)}, T^{(4)}$ that induces 2-edge paths. Otherwise, \ref{itm:case4} holds.
        \end{proof}
        
        An analogous claim that produces subsets $Q_2, R_2, S_2, T_2\subseteq Y$ and vertex $u_2$ also holds for $Y$. We may hence fix the vertex sets $Q_i, R_i, S_i, T_i$ and vertices $u_i$, $i=1,2$, that satisfy the claim above and the analogous statement for $Y$. We then proceed to further refine these subsets.
        
        \begin{claim}
            Let $\ell_1 = \log \ell_0 / 100$. There exist subset $Q_i', R_i', S_i', T_i'$ of $Q_i, R_i, S_i, T_i$, respectively, for $i =1,2$, which are of the same size $\ell_1$ and satisfy one of (the obvious analogues of) \ref{itm:case1} to \ref{itm:case4}, and, additionally, each of $u_i, R_i', S_i', T_i'$ is either complete or anti-complete to $Q_{3-i}$, for $i \in [2]$.
        \end{claim}
        
        \begin{proof}[Proof of the claim]
            For $q \in Q_1$, let $\nr(q)$ be the unique neighbour of $q$ in $R_1$, assuming one of \ref{itm:case2} to \ref{itm:case4} holds; otherwise, choose $\nr(q)$ arbitrarily (say, from $R_1$; it will play no roles). Similarly, let $\ns(q)$ be the unique neighbour of $\nr(q)$ in $S_1$, if one of \ref{itm:case3} and \ref{itm:case4} holds, and let $\nt(q)$ be the unique neighbour of $\ns(q)$ if \ref{itm:case4} holds.
            For the cases when $s(q)$ or $t(q)$ are not defined, we again choose them arbitrarily as they will play no roles.
            For $q \in Q_2$, we define $\nr(q), \ns(q), \nt(q)$ analogously with respect to $Q_2, R_2, S_2, T_2$.
            
            We then consider an auxiliary edge-coloured complete bipartite graph $H$ on $Q_1\cup Q_2$ as follows: for $q_1 \in Q_1$ and $q_2 \in Q_2$, colour $q_1 q_2$ by the $0$-$1$ vector of length eight that encodes which of the pairs in $\{q_{3-i}\} \times \{\nr(q_i), \ns(q_i), \nt(q_i), u_i\}$ are edges in $G$ or not, for $i=1,2$.
            
            By the classical K\"ovari--S\'os--Turan theorem, there exist subsets $Q_1' \subseteq Q_1$ and $Q_2' \subseteq Q_2$ of size $\ell_1$ each, such that $H[Q_1', Q_2']$ is monochromatic. Taking $R_i' := \{\nr(q) : q \in Q_i'\}$, $S_i'=\{\ns(q) : q \in Q_i'\}$, and $T_i'=\{\nt(q) : q \in Q_i'\}$ proves the claim.
        \end{proof}
        
        For brevity, let us rename $Q_i', R_i', S_i', T_i'$ to $Q_i, R_i, S_i, T_i$, respectively, for $i=1,2$.
        
        A simple (though somewhat tedious) case analysis now shows that one of \ref{itm:a} to \ref{itm:f} above holds by using $\ell_1 \ge \log \ell / 10000 = \ell'$. 
        \begin{itemize}
            \item
                {\bf Property \ref{itm:case1} holds for $i =1,2$.}
                
                We claim that one of \ref{itm:a} and \ref{itm:b} holds. Indeed, take $X' := Q_1$ and $Y' := Q_2$. If one of $u_1$ and $u_2$ is complete to both $Q_1$ and $Q_2$ for some $i \in [2]$, then \ref{itm:a} hold. Otherwise, taking $a := u_1$ and $b := u_2$ makes \ref{itm:b} hold.
            \item
                {\bf Property \ref{itm:case1} holds for neither $i =1$ nor $2$.}
                
                We claim that one of \ref{itm:c} and \ref{itm:d} holds. 
                Indeed, if $R_i$ is anti-complete to $Q_{3-i}$ for both $i=1$ and $2$, then \ref{itm:c} holds by taking $X', Y', A, A'$ to be $Q_1, Q_2, R_1, R_2$, respectively). Otherwise, without loss of generality $R_1$ is complete to $Q_2$, and \ref{itm:d} holds by taking $X', Y', A$ to be $Q_1, Q_2, R_1$.
        \end{itemize}
        Without loss of generality, it remains to consider the case where property \ref{itm:case1} holds for $i = 2$ but not for $i = 1$.
        At every step of the following case analysis, we iteratively assume that the previous cases do not hold.
        \begin{itemize}
            \item
                {\bf The vertex $u_2$ is complete to $Q_1$.} Then \ref{itm:a} holds.
            \item
                {\bf The set $R_1$ is complete to $Q_2$.} Then \ref{itm:d} holds.
            \item
                {\bf Property \ref{itm:case2} holds for $i = 1$.} 
                
                Then \ref{itm:e} holds: take $X', Y', A, A', b$ to be $Q_1, Q_2, R_1, \{u_1\}, u_2$.
            \item
                {\bf Property \ref{itm:case3} holds for $i = 1$.}
                
                Again \ref{itm:e} holds: take $X', Y', A, A', b$ to be $Q_1, Q_2, R_1, S_1, u_2$.
            \item
                {\bf Property \ref{itm:case4} holds for $i = 1$.}
                
                Here we may assume that $R_1, S_1, T_1$ are anti-complete to $Q_2$, as otherwise \ref{itm:d} holds.
                It follows that \ref{itm:f} holds: take $X', Y', A, A', A'', b$ to be $Q_1, Q_2, R_1, S_1, T_1, u_2$, respectively. \qedhere
        \end{itemize}
    \end{proof}

    \subsection{Finding a cycle with exactly $k$ chords}
    
    Finally, we prove the following lemma, which finds a cycle with exactly $k$ chords given a large collection of pairwise vertex-disjoint $K_{\ell,\ell}$'s with no edges between them.
    As mentioned previously, we will use a similar approach to that used in the previous section and earlier in this section, connecting the $K_{\ell,\ell}$'s by unimodal paths and closing a cycle by choosing a path of the right length from each $K_{\ell,\ell}$. The difference here is the much more careful analysis of the interaction between a single $K_{\ell,\ell}$, and a collection of parents, `grandparents' and `great-grandparents' of its vertices. This analysis will give rise to three cases (the first corresponding to the first four cases in \Cref{lem:cases} and the last two each corresponding to one of the last two cases in \Cref{lem:cases}). In the last two cases we will sometimes need to adjust the cycle lengths slightly. For that we build our cycle so as to contain certain small `gadgets' that will allows us to do so.  
    
    \begin{lem} \label{lem:find-cycle-from-Kll-exact}
        Let $p, \ell \gg k \gg 1$. Suppose that, in a $p$-extraction of $G$, there are $\ell$ pairwise vertex-disjoint induced copies of $K_{\ell, \ell}$ with no edges between them. Then there is a cycle in $G$ with exactly~$k$ chords.
    \end{lem}
    
    \begin{proof}
        Suppose that there is no cycle with exactly $k$ chords. 
        Let $m$ satisfy $p, \ell \gg m \gg k$. 
        Then by \Cref{lem:cleaning-step-one}, there exist a collection $K_1, \ldots, K_{200}$ of pairwise vertex-disjoint copies of induced $K_{m,m}$'s in the $p$-extraction of $G$, a subset $J \subseteq [p]$ of size $200 \cdot 75$ that satisfy:
        for every triple $T=\{j_1,j_2,j_3\}$ in $J$ with $j_1>j_2>j_3$, every $i\in[200]$ and every $u\in K_i$,
        there exist vertices $f_{1}(u;T)$, $f_{2}(u;T)$,  and $f_{3}(u;T)$
        such that $f_t(u;T)$ is $j_t$-father of $f_{t-1}(u;T)$, $t=1,2,3$, where $u=f_0(u;T)$. Moreover, there are no edges between $\{f_{1}(u;T),f_{2}(u;T),f_{3}(u;T)\}$ and $\cup_{s\neq i}V(K_s)$.
        Let $S_1, \ldots, S_{200}\subseteq J$ be pairwise disjoint subsets of size $75$.

        Fix $i \in [200]$ and let $T_{1}, \ldots, T_{25}$ be disjoint triples that partition $S_i$. Now repeatedly apply \Cref{lem:cases} twenty five times to each $K_i$: at the beginning, let $K=K_i$ with the bipartition $X_i^{(0)}\cup Y_i^{(0)}$. At the $j$-th iteration, we apply \Cref{lem:cases} with $X=X_i^{(j-1)}$, $Y=Y_i^{(j-1)}$, $R=\{f_{1}(u;T_j) : u \in X_i^{(j-1)} \cup Y_i^{(j-1)}\}$, $S=\{f_{2}(u;T_j) : u \in X_i^{(j-1)} \cup Y_i^{(j-1)}\}$ and $T=\{f_{3}(u;T_j) : u \in X_i^{(j-1)} \cup Y_i^{(j-1)}\}$.
        Denote the output of the lemma by $X'_{j}, Y'_{j}, A_{j}, A'_{j}, A''_{j}$ and vertices $a_{j},b_{j}$. Recall that either $X'_{j} \subseteq X$ and $Y'_{j} \subseteq Y$, or $X'_{j} \subseteq Y$ and $Y'_{j} \subseteq X$.
        In the former case, take $X_i^{(j)} = X_{j}'$ and $Y_i^{(j)} = Y_{j}'$ and otherwise take $X_i^{(j)} = Y_{j}'$ and $Y_i^{(j)} = X_{j}'$; so that $X_i^{(j)} \subseteq X_{i-1}^{(j)}$ and $Y_i^{(j)} \subseteq Y_{i-1}^{(j)}$. 

        Notice that there exists $L \subseteq [25]$ of size $3$ such that the same case among \ref{itm:a} to \ref{itm:f} occurred in all iterations indexed by $L$, and moreover either for all $\ell \in L$ we have $X_i^{(\ell)} = X_{\ell}'$ or we always have $X_i^{(\ell)} = Y_{\ell}'$.
		We assume that $X_i^{(\ell)} = X_\ell'$ for all $\ell \in L$; the other case can be handled analogously.
		We consider three cases, as follows.
        For brevity, write $X_i = X_i^{(25)}$ and $Y_i = Y_i^{(25)}$; so $X_i$ and $Y_i$ are the shrunken sets at the end of the 25 iterations, each of which has size at least $(\log ^{(50)} m) / 10^6$.\footnote{Here 
        $\log ^{(\ell)} x$ denotes $\ell$ compositions of logarithm, e.g., $\log^{(2)}x = \log\log x$.}
        \begin{enumerate} [label = \rm(\arabic*)]
            \item \label{itm:easy-case}
				One of \ref{itm:a} to \ref{itm:d} occurred in all iterations indexed by $L$.

                Let $\ell_1, \ell_2 \in L$ be distinct.
                We choose vertices $u_{i,1}, u_{i,2}, u_{i,1}', u_{i,2}'$ as follows.
                \begin{itemize}
                    \item \ref{itm:a}:
                        Take $u_{i, 1} = a_{\ell_1}$ and $u_{i, 2} = a_{\ell_2}$, so that $u_{i,1}$ and $u_{i,2}$ are complete to $X_i \cup Y_i$; choose $u_{i, 1}' \in X_i$ and $u_{i, 2}' \in Y_i$. 
                    \item \ref{itm:b}:
						Take $u_{i,1} = a_{\ell_1}$ and $u_{i,2} = b_{\ell_2}$, so that $u_{i,1}$ is complete to $X_i$ and anti-complete to $Y_i$ and $u_{i, 2}$ is complete to $Y_i$ and anti-complete to $X_i$; choose $u_{i, 1}' \in X_i$ and $u_{i, 2}' \in Y_i$.
                    \item \ref{itm:c}:
						Pick $u_{i,1}' \in X_i$ and $u_{i,2}' \in Y_i$ arbitrarily, and let $u_{i,1}$ be the unique neighbour of $u_{i,1}'$ in $A_{\ell_1}$ and $u_{i,2}$ to be the unique neighbour of $u_{i,2}'$ in $A'_{\ell_2}$.
						Notice that $u_{i,s}'$ is the unique neighbour of $u_{i,s}$ in $X_i \cup Y_i$, for $s \in [2]$.
                    \item \ref{itm:d}:
						Let $u_{i,1}', u_{i,2}' \in X_i$ be distinct, and let $u_{i,s}$ be unique neighbour of $u_{i,s}'$ in $A_{\ell_s}$, for $s \in [2]$.
						Then $u_{i,s}$ is complete to $Y_i$ and $u_{i,s}'$ is its only neighbour in $X_i$, for $s \in [2]$.
                \end{itemize}
				Notice that $u_{i,1},u_{i,2}$ belong to different layers with indices in $S_i$ and are anti-complete to $X_j$ for $j \in [200] \setminus \{i\}$.
            \item \label{itm:annoying-case-a}
				Case \ref{itm:e} occurred in all cases indexed by $L$.

				Write $L = \{\ell_1, \ell_2, \ell_3\}$.
				Let $u_{i,1}', u_{i,3}' \in X_i$ be distinct and let $u_{i,s}$ be the unique neighbour of $u_{i,s}'$ in $A_{\ell_s}$, for $s \in \{1,3\}$. Take $u_{i,4}$ to be a neighbour of $u_{i,3}$ in $A'_{\ell_3}$, take $u_{i,2} = b_{\ell_2}$ and choose $u_{i,2}' \in Y_i$.
				Then
                \begin{itemize}
                    \item
						$u_{i, 1}, \ldots, u_{i, 4}$ belong to four distinct layers with indices in $S_i$ and are anti-complete to $X_j \cup Y_j$, for $j \in [200] \setminus \{i\}$.
                    \item 
						$u_{i, 1}$ and $u_{i, 3}$ are anti-complete to $Y_i$, and $u_{i,s}'$ is the unique neighbour of $u_{i,s}$ in $X_i$, for $s \in \{1,3\}$.
                    \item
                        $u_{i, 2}$ is complete to $Y_i$ and anti-complete to $X_i$. 
                    \item
                        $u_{i, 4}$ is anti-complete to $X_i$ and either complete or anti-complete to $Y_i$.
                \end{itemize}
            \item \label{itm:annoying-case-b}
				Case \ref{itm:f} occurred in all cases indexed by $L$.

				Write $L = \{\ell_1, \ell_2, \ell_3\}$.
				Let $u_{i,1}', u_{i,3}' \in X_i$ be distinct and let $u_{i,s}$ be the unique neighbour of $u_{i,s}'$ in $A_{\ell_s}$, for $s \in \{1,3\}$. Take $u_{i,4}$ to be the unique neighbour of $u_{i,3}$ in $A'_{\ell_3}$ and $u_{i,5}$ the unique neighbour of $u_{i,4}$ in $A''_{\ell_3}$. Finally, let $u_{i,2} = b_{\ell_2}$ and pick $u_{i,2}' \in Y_i$.
				Then
                \begin{itemize}
                    \item
						$u_{i, 1}, \ldots, u_{i, 5}$ are in distinct layers with indices in $S_i$ and are anti-complete to $X_j \cup Y_j$ for $j \in [200] \setminus \{i\}$.
                    \item 
						$u_{i,s}'$ is the unique neighbour of $u_{i,s}$ in $X_i$ and is anti-complete to $Y_i$, for $s \in \{1,3\}$.
                    \item
                        $u_{i, 2}$ is complete to $Y_i$ and anti-complete to $X_i$. 
                    \item
                        $u_{i, 4}$ is a neighbour of $u_{i, 3}$ and $u_{i, 3}'$ and is anti-complete to $(X_i \cup Y_i) \setminus \{u_{i, 3}'\}$.
                    \item
                        $u_{i, 5}$ is a neighbour of $u_{i, 3}'$ and $u_{i, 4}$ and is anti-complete to $(X_i \cup Y_i) \setminus \{u_{i, 3}'\}$.
                \end{itemize}
        \end{enumerate}
        Notice that one of the following holds: \ref{itm:easy-case} holds for $20$ values of $i \in [200]$; \ref{itm:annoying-case-a} holds for $100$ values of $i \in [200]$; or \ref{itm:annoying-case-b} holds for $80$ values of $i \in [200]$.
		We show how to find the desired cycle with $k$ chords in each of these cases separately.
        
        \medskip
        
        \textbf{Case \ref{itm:easy-case} holds for at least $20$ values of $i \in [200]$.}
        
        By relabelling the indices of the $200$ copies of $K_{m,m}$ if necessary, we may assume that \ref{itm:easy-case} holds for $i \in [20]$. To construct a cycle with $k$ chords, we need one more copy of $K_{m,m}$, so we shall use $K_{21}$ too. For simplicity, let $K_0=K_{21}$ be the copy of $K_{m,m}$ on $X_0\cup Y_0$, where $X_0=X_{21}$ and $Y_0=Y_{21}$.
        
		We want to take a set of distinct vertices $w_{i,j}$, $i\in[20]$ and $j \in [2]$, where each $w_{i,j}$ is at the same layer as $u_{i,j}$ and has a neighbour $w_{i,j}'$ in $K_0$ which are all distinct too, and there are no edges between $\bigcup_{i \in [20]} (X_i \cup Y_i)$ and $\{w_{i,j}: i \in [20], j \in [2]\}$. Here the index $j=1,2$ indicates $w_{i,1}'\in X_0$ and $w_{i,2}'\in Y_0$, as was done for $u_{i,j}'$'s.
        Indeed, this is possible since one can greedily take $w_{i,1}'$ and $w_{i,2}'$ disjoint from the previous choices in $X_0$ and $Y_0$, respectively, and let $w_{i,j}=f_1(w'_{i,j};T)$ where $T$ is any triple whose smallest index is the index of the layer containing $u_{i,j}$.
        Let $P_{i, j}$ be a unimodal path between $u_{i, j}$ and $w_{i, j}$, taken in the same layers as the two end vertices so that they are in different layers from the $K_i$'s. 

		Given integers $\alpha_1, \ldots, \alpha_{20}$, we define a cycle $\C(\alpha_1, \ldots, \alpha_{20})$ as follows. Let $Q_i$ be the path from $w_{i, 1}'$ to $w_{i+1, 1}$ for $i \in [20]$ (addition of indices taken modulo $20$), obtained by concatenating the paths $w_{i, 1}' w_{i, 1}$, $P_{i,1}$, $u_{i, 1} u_{i, 1}'$, $Q_i^*$, $u_{i, 2}' u_{i, 2}$, $P_{i,2}$, $w_{i, 2} w_{i, 2}' w_{i+1, 1}'$, where $Q_i^*$ is a path in $G[X_i, Y_i]$ with ends $u_{i, 1}'$ and $u_{i, 2}'$ of length $2\alpha_i + 1$. 
        Now let $\C(\alpha_1, \ldots, \alpha_{20})$ be the cycle obtained by concatenating $Q_1, \ldots, Q_{20}$.
        
        \begin{figure}[h]
            \centering
            \includegraphics{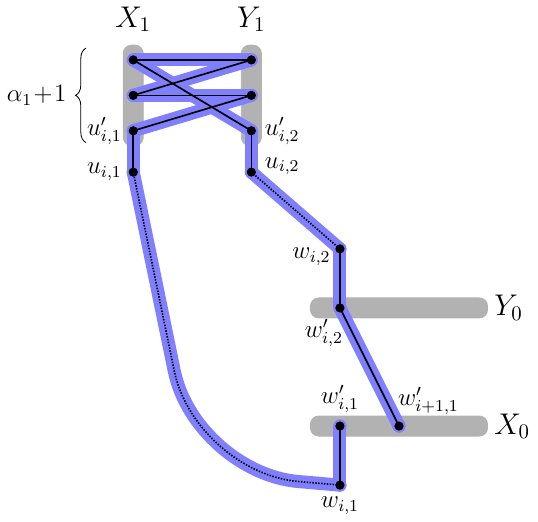}
            \caption{The path $Q_i$ in Case \ref{itm:easy-case}}
            \label{fig:easy-case}
        \end{figure}
        
        While the cycle $\C(\alpha_1, \ldots, \alpha_{20})$ is not uniquely defined, its number of chords depends only on $\alpha_1, \ldots, \alpha_{20}$. 
        To evaluate the number of chords in this cycle, $Q_i^*$ contributes $(\alpha_i + 1)^2 - (2\alpha_i + 1) = \alpha_i^2$ chords; the number of chords with one end in $\{u_{i, 1}, u_{i, 2}\}$ and one end in $X_i \cup Y_i$ is one of the four values $2(2\alpha_i + 1)$, $2\alpha_i$, $0$, and $2(\alpha_i + 1)$, depending on the four cases described in~\ref{itm:easy-case}; there are at most $120^2$ chords with both ends in $\bigcup_{i \in [20], j \in [2]} \{w_{i, j}, w_{i, j}', u_{i, j}\}$; and we may assume that the number of chords with ends in $\bigcup_{i \in [20], j \in [2]} V(P_{i,j})$ is $O(\sqrt{k})$ by \Cref{lem:sqrtk-different-Kll,lem:sqrtk-same-Kll}. Note that there are no chords between distinct sets $X_i \cup Y_i$, between a set $X_i \cup Y_i$ and $\{u_{j,1}, u_{j,2}, w_{j, 1}, w_{j, 2}\}$ for distinct $i, j \in [20]$, or between $\{u_{i,1}, u_{i, 2}\}$ and $\{w_{j, 1}, w_{j,1}', w_{j,2}, w_{j,2}'\}$.  In total, the number of chords is 
        \begin{equation*}
            f_G(k) + \sum_{i \in [20]} (\alpha_i^2 + 2\sigma_i \alpha_i) 
            = f_G(k) + \sum_{i \in [20]} \left((\alpha_i + \sigma_i)^2 - \sigma_i^2\right) 
            = h_G(k) + \sum_{i \in [20]} (\alpha_i + \sigma_i)^2.
        \end{equation*}
        where $\sigma_i \in \{0, 1, 2\}$ for $i \in [20]$, and $f_G(k)$ and $h_G(k)$ are functions that do not depend on $\alpha_1, \ldots, \alpha_{20}$ and satisfy $f_G(k), h_G(k) = O(\sqrt{k})$. By \Cref{lem:twenty-squares}, there is a choice of integers $\beta_1 \ldots, \beta_{20} \ge 2$ such that $\sum_{i \in [20]} \beta_i^2 = k - h_G(k)$. A cycle $\C(\alpha_1, \ldots, \alpha_{20})$ with $\alpha_i = \beta_i - \sigma_i$ has $k$ chords, as desired.
        
        \medskip
        
        \textbf{Case \ref{itm:annoying-case-a} holds for at least $100$ values of $i \in [200]$.}
        
        Again by possible relabelling, we may assume that \ref{itm:annoying-case-a} holds for $i \in [100]$.
        We also need an extra copy, denoted by $K_0$, taken from $K_i$, $i>100$, on the bipartition $X_0\cup Y_0$.
        Let $w_{i, j}$, with $i \in [100]$ and $j \in [4]$, satisfy the following: $w_{i, j}$ is at the same layer as $u_{i, j}$; it is anti-complete to $X_i \cup Y_i$ for $i \in [100]$; $w_{i, j}$ has a neighbours $w_{i, j}'$ such that $w_{i, j}' \in X_0$ if $j \in \{1, 3\}$ and otherwise $w_{i, j}' \in Y_0$; and the vertices $w_{i, j}, w_{i, j}'$, with $j \in [4]$ and $i \in [100]$, are all distinct. 
        Indeed, such choices of $w_{i,j}$ and $w_{i,j}'$ are possible since one can again take $w_{i,j}=f_1(w_{i,j}';T)$ with any triple $T$ whose smallest index is the index of the layer containing $u_{i,j}$, while maintaining all $w_{i,j}'$ being distinct.
        
        Let $P_{i, j}$ be a unimodal path with ends $u_{i, j}$ and $w_{i, j}$ in the same layer as the end vertices,
        for each $i \in [100]$ and $j \in [4]$. 
        Given positive integers $\alpha_1, \ldots, \alpha_{100}$, we define a cycle $\C(\alpha_1, \ldots, \alpha_{100})$ as follows. Let $Q_i$ be the path from $w_{i, 1}'$ to $w_{i+1, 1}'$ (index addition taken modulo $100$), obtained by concatenating the paths $w_{i, 1}' w_{i, 1}$, $P_{i, 1}$, $u_{i, 1} u_{i, 1}'$, $Q_i^*$, $u_{i, 2}'u_{i, 2}$, $P_{i, 2}$, $w_{i, 2} w_{i, 2}' w_{i, 3}' w_{i, 3}$, $P_{i,3}$, $u_{i, 3} u_{i, 4}$, $P_{i, 4}$, $w_{i, 4} w_{i, 4}' w_{i+1, 1}'$, where $Q_{i}^*$ is a path in $G[X_i \setminus \{u_{i, 3}'\}, Y_i]$ with ends $u_{i, 1}'$ and $u_{i, 2}'$, and of length $2\alpha_i + 1$. Let $\C(\alpha_1, \ldots, \alpha_{100})$ be the cycle obtained by concatenating $Q_1, \ldots, Q_{100}$.
        
        \begin{figure}
            \centering
            \includegraphics[scale = .8]{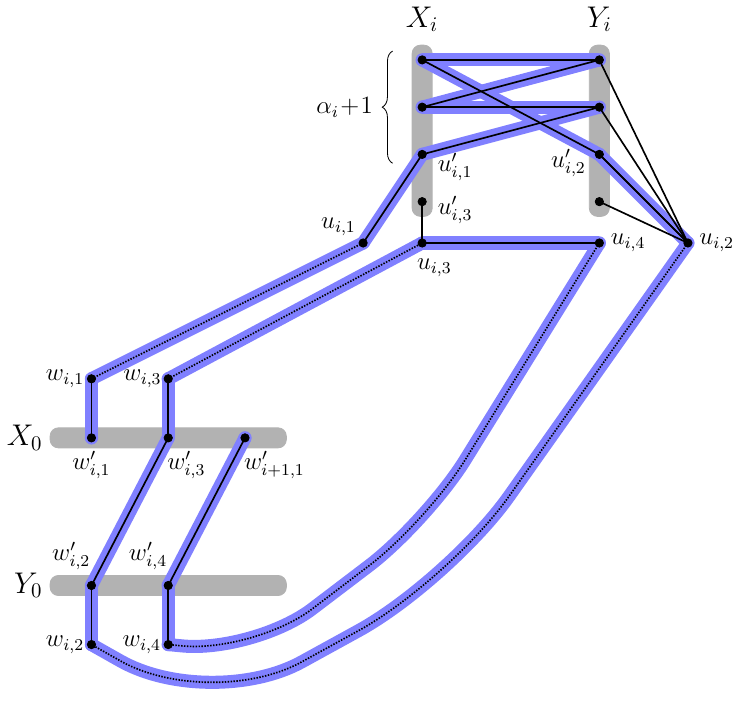}
            \caption{The path $Q_i$ in Case \ref{itm:annoying-case-a}}
            \label{fig:annoying-case-a}
        \end{figure}
        
        As above, the cycle $\C(\alpha_1, \ldots, \alpha_{100})$ is not uniquely defined, but the number of its chords depends only on $\alpha_1, \ldots, \alpha_{100}$. 
        Let us evaluate the number of chords precisely as follows: 
        $Q_{i}^*$ contributes $\alpha_i^2$ chords; 
        the number of chords between $\{u_{i, 1}, \ldots, u_{i, 4}\}$ and $X_i \cup Y_i$ is either $\alpha_i$ or $2\alpha_i+1$ depending on how $u_{i,4}$ connects to $Y_i$ (either complete or anti-complete); there are at most $1200^2$ chords with ends in $\bigcup_{i \in [100], j \in [4]}\{w_{i, j}, w_{i, j}', u_{i, j}\}$; and we may assume that the number of chords in $\bigcup_{i \in [100], j \in [4]} V(P_{i, j})$ is $O(\sqrt{k})$ by \Cref{lem:sqrtk-different-Kll,lem:sqrtk-same-Kll}. In total, the number of chords in the cycle is
        \begin{align*}
            f_G(k) + \sum_{i \in [100]}(\alpha_i^2 + \sigma_i \alpha_i),
        \end{align*}
        where $\sigma_i \in \{1, 2\}$ and $f_G(k) = O(\sqrt{k})$. Here $\sigma_i$'s and $f_G(k)$ do not depend on $\alpha_i$'s.
        
        Suppose that $\sigma_i = 2$ for at least $20$ values of $i \in [100]$; let $I$ be a set of $20$ such values of $i$. Set $\alpha_i = 1$ for $i \in [100] \setminus I$. Each $\alpha_i$, $i \in I$, is not yet determined. Then the number of chords in a cycle $\C(\alpha_1, \ldots, \alpha_{100})$ reduces to 
        \begin{equation*}
            h_G(k) + \sum_{i \in I} (\alpha_i + 1)^2,
        \end{equation*}
        for some $h_G(k) = O(\sqrt{k})$ that does not depend on $\alpha_i$'s with $i \in I$. By \Cref{lem:twenty-squares}, there exist integers $\beta_i$, for $i \in I$, such that $\beta_i \ge 2$ and $\sum_{i \in I} \beta_i^2 = k - h_G(k)$. Set $\alpha_i = \beta_i - 1$. The cycle $\C(\alpha_1, \ldots, \alpha_{100})$ then has exactly $k$ chords.
        
        It remains to consider the case where $\sigma_i = 1$ for at least $80$ values of $i \in [100]$. Let $I$ be a set of $80$ values of $i$ such that $\sigma_i=1$. Set $\alpha_i = 1$ for $i \in [100] \setminus I$. Then the number of chords in the cycle $\C(\alpha_1, \ldots, \alpha_{100})$ is 
        \begin{equation*} 
             h_G(k) + \sum_{i \in I} (\alpha_i^2 + \alpha_i),
        \end{equation*}
        where $h_G(k) = O(\sqrt{k})$ is independent of the $\alpha_i$'s. Let $\tau \in \{0, 1, 2, 3\}$ be the remainder of $k - h_G(k)$ modulo $4$. Let $\C_{\tau}(\alpha_1, \ldots, \alpha_{100})$ be a cycle obtained by replacing an inner vertex of $V(Q_i^*) \cap X_i$ by $u_{i, 3}'$, for $\tau$ values of $i$. One can readily check that the number of chords in $\C_{\tau}(\alpha_1, \ldots, \alpha_{100})$ is $h_G(k) + \tau + \sum_{i \in I} (\alpha_i^2 + \alpha_i)$m as we gain the additional chord $u_{i, 3} u_{i, 3}'$ for $\tau$ values of $i$. Now, since $k - h_G(k) - \tau$ is divisible by $4$, \Cref{lem:sum-skewed-squares} implies that there exist integers $\alpha_i$, $i \in I$, such that $\alpha_i \ge 1$ and $\sum_{i \in I} (\alpha_i^2 + \alpha_i) = k - h_G(k) - \tau$. In particular, there is a cycle with exactly $k$ chords.
        
        \medskip
        
        \textbf{Case \ref{itm:annoying-case-b} holds for at least $80$ values of $i \in [200]$}
        
        Without loss of generality, \ref{itm:annoying-case-b} holds for $i \in [80]$ and let $K_0=K_{81}$ be another copy of $K_{m,m}$ on $X_0\cup Y_0$. Analogously to the previous cases, choose $w_{i, j}$, $i \in [80]$ and $j \in [5]$, that satisfy the following conditions: $w_{i, j}$ is in the same layer as $u_{i, j}$; it is anti-complete to $X_i \cup Y_i$ for $i \in [80]$; $w_{i, j}$ has a neighbour $w_{i, j}'$ such that $w_{i, j}' \in X_0$ for $j \in \{1, 3\}$ and $w_{i, j}' \in Y_0$ for $j \in \{2, 4, 5\}$; and the vertices $w_{i, j}, w_{i,j}'$, with $i \in [80]$ and $j \in [5]$, are all distinct. 
        As before, such choices for $w_{i,j}$ and $w_{i,j}'$ exist.
        
        One can find a path $P_{i, j}$ between $u_{i, j}$ and $w_{i, j}$ in the same layer as the two ends, for each $i \in [80]$ and $j \in [5]$. 
        In fact, we do not make use of $w_{i, 4}$, $w_{i, 4}'$ and $P_{i, 4}$, but we took them to preserve the correspondence between indices.
        
        Given positive integers $\alpha_1, \ldots, \alpha_{80}$, we choose a cycle $\C^0(\alpha_1, \ldots, \alpha_{80})$ as follows. Let $Q_i$ be the path from $w_{i, 1}'$ to $w_{i+1, 1}'$ (addition modulo $80$), obtained by concatenating the paths $w_{i,1}' w_{i,1}$, $P_{i, 1}$, $u_{i, 1} u_{i, 1}'$, $Q_{i}^*$, $u_{i, 2}' u_{i, 2}$, $P_{i, 2}$, $w_{i, 2} w_{i, 2}' w_{i, 3}' w_{i, 3}$, $P_{i, 3}$, $u_{i, 3} u_{i, 4} u_{i, 5}$, $P_{i, 5}$, $w_{i, 5} w_{i, 5}' w_{i+1, 1}'$; where $Q_i^*$ is a path in $G[X_i \setminus \{u_{i, 3'}\}, Y_i]$ with ends $u_{i,1}'$ and $u_{i, 2}'$ and of length $2\alpha_i + 1$. Let $\C^0(\alpha_1, \ldots, \alpha_{80})$ be the cycle obtained by concatenating $Q_1, \ldots, Q_{80}$. For $\tau =1,2,3$, let $\C^{\tau}(\alpha_1, \ldots, \alpha_{80})$ be the cycle obtained from $\C^0(\alpha_1, \ldots, \alpha_{80})$ by replacing one of the internal vertices of $Q_{i}^*$ in $X_i$ by $u_{i, 3}'$, for $i \in [\tau]$. 
        \begin{figure}
            \centering
            \includegraphics[scale = .8]{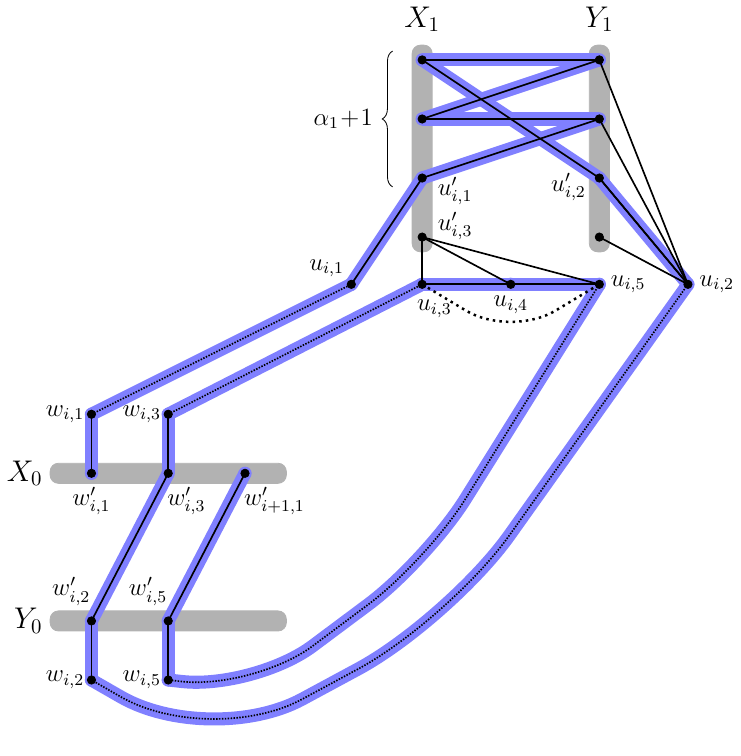}
            \caption{The path $Q_i$ in Case \ref{itm:annoying-case-b}}
            \label{fig:annoying-case-b}
        \end{figure}
        
        The cycle $\C^\sigma(\alpha_1, \ldots, \alpha_{80})$ for each $\sigma=0,1,2,3$ is not necessarily unique, but all the choices have the same number of chords. 
        To verify, let us evaluate the number of chords in $\C^0(\alpha_1, \ldots, \alpha_{80})$ first, as follows: $Q_i^*$ contributes $\alpha_i^2$ chords; there are exactly $\alpha_i$ chords between $\{u_{i,1}, \ldots, u_{i, 5}\}$ and $X_i \cup Y_i$, all of which are incident to $u_{i,2}$; there are at most $1200^2$ chords with ends in $\bigcup_{i \in [80], j \in [5]} \{w_{i, j}, w_{i,j}', u_{i, j}\}$; and we may assume that the number of chords in $\bigcup_{i \in [80], j \in [5]} V(P_{i, j})$ is $O(\sqrt{k})$, by \Cref{lem:sqrtk-different-Kll,lem:sqrtk-same-Kll}. 
        The cycle $\C^{\tau}(\alpha_1, \ldots, \alpha_{80})$ has precisely $3\tau$ more chords than $\C^0(\alpha_1, \ldots, \alpha_{80})$, as $u_{i,3}'$ gives three extra chords to $u_{i,3}$, $u_{i,4}$, and $u_{i,5}$ (see \Cref{fig:annoying-case-b}).
        The total number of chords in $\C^{\tau}(\alpha_1, \ldots, \alpha_{80})$ is hence
        \begin{equation*}
            f_G(k) + \sum_{i \in [80]}(\alpha_i^2 + \alpha_i) + 3\tau,
        \end{equation*}
        where $f_G(k) = O(\sqrt{k})$. 
        Now choose $\tau \in \{0, 1, 2, 3\}$ be such that $3\tau$ equals $k - f_G(k)$ modulo $4$. By \Cref{lem:sum-skewed-squares}, there exist positive integers $\alpha_1, \ldots, \alpha_{80}$, such that $\sum_{i \in [80]} \alpha_i(\alpha_i + 1) = k - f_G(k) - 3\tau$. The cycle $\C^{\tau}(\alpha_1, \ldots, \alpha_{80})$ has exactly $k$ chords.
    \end{proof}

\section{Conclusion} \label{sec:conc}
    
    We showed that the family of graphs with no cycle with exactly $k$ chords is $\chi$-bounded, for every integer $k$ which is either sufficiently large or of form $k = \ell(\ell-2)$, where $\ell \ge 3$ is an integer. This was already known to hold for $k \in \{1,2,3\}$ (see \cite{trotignon2010structure,aboulker2015excluding}).
    An obvious follow-up problem, which is a conjecture due to Aboulker and Bousquet \cite{aboulker2015excluding}, would be to extend this to all $k \ge 1$.
    
    \begin{conj}[Aboulder--Bousquet \cite{aboulker2015excluding}]\label{Conjecture_Aboulder_Bousquet}
        For \emph{every} $k \ge 1$ there is a function $f_k$ such that, if $G$ is a graph with no cycle with exactly $k$ chords, then $\chi(G) \le f_k(\omega(G))$.
    \end{conj}
    
  It would also be interesting to get a better understanding of whether graphs with small clique number and large chromatic number contain cycles with $k$ chords that have some sort of further structure. Whereas our proof essentially just produces a cycle with $k$ chords, there are conjectures that it should be possible to find more. A $k$-fan is defined as a path $P$ with one additional vertex $v$ added that has exactly $k$ neighbours on $P$. It is easy to see that $k$-fans contain a cycle with $k-2$ chords. Thus the following would be a strengthening of the results in this paper and of Conjecture~\ref{Conjecture_Aboulder_Bousquet}.
	    \begin{conj}[Davies \cite{davies2023triangle}]
        For \emph{every} $k \ge 1$ there is a function $f_k$ such that, if $G$ is a graph with $k$-fan, then $\chi(G) \le f_k(\omega(G))$.
    \end{conj}	
    Currently, this conjecture is only known to hold for $k=1$ (where is it equivalent to the ``$k=1$'' case of Conjecture~\ref{Conjecture_Aboulder_Bousquet}).

\paragraph{Acknowldegments}
We are grateful to anonymous referees for giving careful comments, especially on the number of iterations in \Cref{lem:find-cycle-from-Kll-exact}.

\bibliography{chi}
\bibliographystyle{amsplain}
\end{document}